\documentclass[a4,12pt]{article}

\usepackage[T1]{fontenc}
\usepackage[latin1]{inputenc}
\usepackage{amsfonts} 
\usepackage{amssymb}
\usepackage{amsbsy}
\usepackage{eucal}
\usepackage{amsmath}
\usepackage{makeidx}
\usepackage{amsthm}
\usepackage{amsxtra}
\usepackage{marvosym}
\makeindex                       
\usepackage[hang,small,bf,it]{caption}
\usepackage[T1]{fontenc}
\usepackage[latin1]{inputenc}
\usepackage{amsfonts} 
\usepackage{amssymb}
\usepackage{amsbsy}
\usepackage{eucal}
\usepackage{amsmath}
\usepackage{makeidx}
\usepackage{amsthm}
\usepackage{amsxtra}
\usepackage{marvosym}
\usepackage[english]{babel}
\usepackage{latexsym,amssymb}
\usepackage{xcolor}
\usepackage[pdftex]{graphicx}
\usepackage{amsmath,amsfonts}
\usepackage{tikz}

\usepackage{latexsym,amssymb}
\usepackage{xcolor}
\usepackage[pdftex]{graphicx}
\usepackage{amsmath,amsfonts,amsthm}
\usepackage{multicol}
\usepackage{pifont}
\usepackage{geometry}
\usepackage{colortbl}
\usepackage[all]{xy}

\usepackage{tikz}
\usetikzlibrary{matrix,arrows}
\usetikzlibrary{shapes,arrows}

\usepackage{enumitem}

\theoremstyle{plain}
\newtheorem{thm}{Theorem}
\newtheorem{lem}[thm]{Lemma}

\newtheorem{prop}[thm]{Proposition}

\newtheorem{setting}[thm]{Setting}

\theoremstyle{definition}

\theoremstyle{definition}
\newtheorem{defi}[thm]{Definition}

\newcommand{\R}{\mathbb{R}}
\newcommand{\Z}{\mathbb{Z}}
\newcommand{\N}{\mathbb{N}}
\renewcommand{\H}{\mathbb{H}}
\newcommand{\C}{\mathbb{C}}

\newcommand*{\rom}[1]{\romannumeral}
\renewcommand{\Re}{{Re}}

\DeclareMathOperator{\GL}{GL}
\DeclareMathOperator{\SO}{SO}

\DeclareMathOperator{\Spin}{Spin}

\DeclareMathOperator{\Tr}{Tr}
\DeclareMathOperator{\tr}{tr}
\DeclareMathOperator{\End}{End}
\DeclareMathOperator{\Vol}{Vol}
\DeclareMathOperator{\rank}{rank}
\DeclareMathOperator{\Id}{Id}
\DeclareMathOperator{\Ad}{Ad}
\DeclareMathOperator{\ad}{ad}

\DeclareMathOperator{\Norm}{Norm}
\DeclareMathOperator{\cent}{Centr}
\DeclareMathOperator{\spec}{spec}

\DeclareMathOperator{\sign}{sign}

\DeclareMathOperator{\prim}{prime}

\DeclareMathOperator{\Cl}{Cl}

\numberwithin{equation}{section}
\numberwithin{thm}{section}

\title{\huge{Twisted Dirac operators and dynamical zeta functions}}

\author{Polyxeni Spilioti}
\date{}
\begin{document}

\maketitle

\textbf{Abstract}. In this paper, we consider the dynamical zeta functions of Ruelle and Selberg associated with the geodesic flow of 
a compact hyperbolic odd dimensional manifold $X$. These functions are initially defined on one complex variable $s$ in some right half-plane of $\C$.
Our goal is the continue meromorphically the dynamical zeta functions to the whole complex plane, using the Selberg trace formula for arbitrary, not necessarily unitary,
representations $\chi$ of the fundamental group.
First, we prove a trace formula for the integral operator $D^{\sharp}_{\chi}(\sigma)e^{-t(D^{\sharp}_{\chi}(\sigma))^{2}}$, induced by the twisted Dirac operator $D^{\sharp}_{\chi}(\sigma)$
on $X$.
Then we use these results to establish the meromorphic continuation of the dynamical zeta functions to $\C$.
%and prove that they admit functional equations relating their values in $s$ with those
%in $-s$.
%Further, we investigate the connection between the Ruelle zeta function at the 
%central point $s=0$ with the refined analytic torsion as it is introduced by Braverman and 
%Kappeler in \cite{BK2}.


\tableofcontents

\begin{center}
\section{\textnormal{Introduction}}
\end{center}

The zeta functions of Ruelle and Selberg are dynamical functions associated with the geodesic flow of a compact hyperbolic manifold $X$.
They can be represented by infinite products, which involve exponentials
of the lengths of the prime closed geodesics on $X$.

Our geometrical setting consists of 
compact hyperbolic manifolds $X$ of odd dimension $d$, obtained as $X:=\Gamma\backslash \H^{d}$, where 
$\H^{d}$ denotes the $d$-dimensional real hyperbolic space and 
$\Gamma$ a discrete torsion-free subgroup of $\SO^{0}(d,1)$.
Here, we identify  $\H^{d}$ with $\SO^{0}(d,1)/ \SO(d)$.
We set $G:=\SO^{0}(d,1)$ and $K:=\SO(d)$. Let $G=KAN$ be the Iwasawa
decomposition of $G$ and $M$ the centralizer of $A$ in $K$.
Let $\widehat{M}$ be the set of the equivalent classes of unitary irreducible representations of $M$. 

We denote by $[\gamma]$ the $\Gamma$-conjugacy class
of $\gamma\in\Gamma$. One can associate to  $[\gamma]\neq e$, a unique closed geodesic 
$c_{\gamma}$. We consider the lengths $l(\gamma)$ of
the corresponding closed geodesics
$c_{\gamma}$.
We call the conjugacy class $[\gamma]$ prime if 
there exist no $k\in\N$ with $k>1$ and $\gamma_{0}\in\Gamma$ such that $\gamma=\gamma_{0}^{k}$.
The prime geodesics can be viewed as the
geodesics of minimum length.

Let $\sigma\in\widehat M$. 
We give here the definition of the zeta functions in terms of 
finite dimensional representations $\chi$ of $\Gamma$.
The twisted Selberg zeta function $Z(s;\sigma,\chi)$ is defined by 
\begin{equation*}
Z(s;\sigma,\chi):=\prod_{\substack{[\gamma]\neq e,\\ [\gamma]\prim}} \prod_{k=0}^{\infty}\det\Big(\Id-\big(\chi(\gamma)\otimes\sigma(m_\gamma)\otimes S^k(\Ad(m_\gamma a_\gamma)_{\overline{\mathfrak{n}}})\big)e^{-(s+|\rho|)l(\gamma)}\Big),
\end{equation*}
where $\overline{\mathfrak{n}}$ is the sum of the negative root spaces of the system $(\mathfrak{g},\mathfrak{a})$
and $S^k(\Ad(m_\gamma a_\gamma)_{\overline{\mathfrak{n}}})$ denotes the $k$-th
symmetric power of the adjoint map $\Ad(m_\gamma a_\gamma)$ restricted to $\overline{\mathfrak{n}}$.
$Z(s;\sigma,\chi)$ is defined for $s\in\C$ with Re$(s)>c_{1}$, for some constant $c_{1}>0$.

We define also the twisted Ruelle zeta function $R(s;\sigma,\chi)$ by 
\begin{equation*}
 R(s;\sigma,\chi):=\prod_{\substack{[\gamma]\neq{e}\\ [\gamma]\prim}}\det(\Id-\chi(\gamma)\otimes\sigma(m_{\gamma})e^{-sl(\gamma)})^{(-1)^{d-1}},
\end{equation*}
where $s\in\C$ with Re$(s)>c_{2}$, for some constant $c_{2}>0$.

The meromorphic continuation of the Ruelle and Selberg zeta function has been long studied by Fried (\cite{Fried}), Bunke and Olbrich (\cite{BO}), Wotzke (\cite{Wo}) and Pfaff
(\cite{Pf}), but only for 
specific representations of the subgroup $\Gamma$.
Under certain assumptions for the representation of $\Gamma$ i.e., the representation is orthogonal and acyclic, 
Fried (\cite[Theorem 1]{Fried}) managed to prove that the Ruelle zeta function admits a meromorphic continuation to the whole complex plane and moreover that 
is regular at zero and its absolute value at zero is equal to the Ray-Singer analytic torsion.
Bunke and Olbrich (\cite{BO}) dealt with the meromorphic continuation of the dynamical zeta functions but in the case of a unitary representation 
of $\Gamma$. Wotzke (\cite{Wo}) generalized in his thesis the results of Fried to the case of the induced representations of $\Gamma$, i.e., representations
arising from restrictions of finite dimensional representations of $G=\SO^{0}(d,1)$. 
Pfaff (\cite{Pf}) extended these results to the co-finite case using the induced representation as well. 

In \cite{Spil}, the existing results for the meromorphic continuation of the dynamical zeta function are generalized to the case of an arbitrary representation of $\Gamma$,
not necessarily unitary. 
In this paper, we study further this case, under the condition that the action of the restricted Weyl group $W_{A}$ on $\widehat{M}$
is not trivial, i.e., $w\sigma\neq\sigma$, where $w$
is a non-trivial element of $W_{A}$.
We have then to introduce three more zeta functions the symmterized $S(s,\sigma,\chi)$, super $Z^{s}(s,\sigma,\chi)$ and super Ruelle $R^{s}(s\,\sigma,\chi)$ zeta function
(see Definitions 3.3, 3.4 and 3.5). To be able to prove the meromorphic continuation of the zeta functions, we introduce the twisted Dirac operator, defined in a 
similar way as the twisted Bochner-Laplace operator in \cite{M1}.
 
We give here a short description of the twisted Dirac operator. As before, we consider an arbitrary
finite dimensional representation 
$\chi$ of $\Gamma$. Let $E_{\chi}$ be the flat vector bundle over $X$.
Let $\widehat{K}$ be the set of the equivalent classes of unitary irreducible representations of $K$. 
Let  $s$ be the spin representation of $K$ and $\tau(\sigma)\in\widehat{K}$. 
We define the representation $\tau_{s}(\sigma)$ of $K$ by $\tau_{s}(\sigma):= s\otimes\tau(\sigma)$.
We consider then the  locally homogeneous vector bundle $E_{\tau_{s}(\sigma)}$ over $X$. 
We let $D(\sigma)$ be the Dirac operator associated with $\tau_{s}(\sigma)$.
Let $D^{\sharp}_{\chi}(\sigma)$ be the twisted Dirac operator acting on $C^{\infty}(X,E_{\tau_{s}(\sigma)}\otimes E_{\chi})$.
Locally, it is given by the following formula.
\begin{equation*}
\widetilde{D}^{\sharp}_{\chi}(\sigma)=\widetilde{D}(\sigma)\otimes \Id_{V_{\chi}},
\end{equation*}
where $\widetilde{D}^{\sharp}_{\chi}(\sigma)$ and $\widetilde{D}(\sigma)$  
are the lifts to $\widetilde{X}$ of 
$D^{\sharp}_{\chi}(\sigma)$ and $D(\sigma)$, respectively. 
Since the representation of $\Gamma$
is not necessarily unitary, we deal with non self-adjoint operators. 
Nevertheless, one can observe by the formula above that the twisted Dirac operator $D^{\sharp}_{\chi}(\sigma)$
has the same principal symbol as ${D}(\sigma)$.
Hence, it has nice spectral properties, i.e., its the spectrum is a discrete subset of a positive cone in $\C$. 
Let $D^{\sharp}_{\chi}(\sigma)e^{-tD^{\sharp}_{\chi}(\sigma)}$ be the associated operator acting on the space
of smooth sections of the vector bundle $E_{\tau_{s}(\sigma)}\otimes E_{\chi}$.
By \cite[Lemma 2.4]{M1}, this is an integral operator with smooth kernel and hence of trace class.
We derive a corresponding trace formula for $D^{\sharp}_{\chi}(\sigma)e^{-tD^{\sharp}_{\chi}(\sigma)}$.
\newline
\begin{thm}
 For every $\sigma \in \widehat{M}$,
 \begin{equation*}
 \Tr(D^{\sharp}_{\chi}(\sigma)e^{-t(D^{\sharp}_{\chi}(\sigma))^{2}})=
  \sum_{[\gamma]\neq e}\frac{-2\pi i}{(4\pi t)^{3/2}}\frac{l^{2}(\gamma)\tr(\chi(\gamma)\otimes(\sigma(m_{\gamma})-w\sigma(m_{\gamma})))}
  {n_{\Gamma}(\gamma)D(\gamma)}e^{-l^{2}(\gamma)/4t}.
 \end{equation*}
\end{thm}
This trace formula is the key point to prove the meromorphic continuation of the super zeta function.
\begin{thm}
The super zeta function $Z^{s}(s;\sigma,\chi)$ admits a meromorphic continuation to the whole complex plane $\C$. The singularities 
are located at $\{s_{k}^{\pm}=\pm i\lambda_{k}:\lambda_{k}\in \spec(D^{\sharp}_{\chi}(\sigma)), k\in\N\}$
of order  $\pm m_{s}(\lambda_{k})$, where $m_{s}(\lambda_{k})=m(\lambda_{k})-m(-\lambda_{k})\in\N$ and 
$m(\lambda_{k})$ denotes the algebraic multiplicity of the eigenvalue $\lambda_{k}$.
\end{thm}
In \cite{Spil}, the differential operator $A^{\sharp}_{\chi}(\sigma)$, which is induced by the twisted Bochner-Laplace operator $\Delta^{\sharp}_{\tau,\chi}$, is introduced.
Moreover, a trace formula for the corresponding heat semi-group $e^{-tA^{\sharp}_{\chi}(\sigma)}$ is proved. In the present case, where $w \sigma\neq \sigma$,
we derive a trace formula for for the operator $e^{-tA^{\sharp}_{\chi}(\sigma)}$, which
is  slightly different.

\begin{thm} 
 For every $\sigma \in \widehat{M}$ we have
   \begin{align*}\label{f:trace formula 2}
  \Tr(e^{-tA_{\chi}^{\sharp}(\sigma)})=&2\dim(V_{\chi})\Vol(X)\int_{\R}e^{-t\lambda^{2}}P_{\sigma}(i\lambda)d\lambda\\
  &+\sum_{[\gamma]\neq e}\frac{l(\gamma)}{n_{\Gamma}(\gamma)}L_{sym}(\gamma;\sigma+w\sigma)
  \frac{e^{-l(\gamma)^{{2}}/4t}}{(4\pi t)^{1/2}},
 \end{align*} 
where 
\begin{equation*}
 L_{sym}(\gamma;\sigma)= \frac{\tr(\sigma(m_{\gamma})\otimes\chi(\gamma))e^{-|\rho|l(\gamma)}}{\det(\Id-\Ad(m_{\gamma}a_{\gamma})_{\overline{\mathfrak{n}})}}.
\end{equation*}
\end{thm}
The trace formula in Theorem 1.3 is again the tool which we use to prove
the meromorphic continuation of the symmetrized zeta function.
\begin{thm}
The symmetrized zeta function $S(s;\sigma,\chi)$ admits a meromorphic continuation to the whole complex plane $\C$. The set of the singularities
equals $\{s_{k}^{\pm}=\pm i{\sqrt{\mu_{k}}}:\mu_{k}\in \spec(A^{\sharp}_{\chi}(\sigma)), k\in\N\}$.
The orders of the singularities are equal to $m(\mu_{k})$, where $m(\mu_{k})\in\N$ denotes the algebraic multiplicity of the eigenvalue $\mu_{k}$.
For $\mu_{0}=0$, the order of the singularity $s_{0}$ is equal to $2m(0)$.
\end{thm}
Finally, we prove the following theorems.
\begin{thm}
The Selberg zeta function $Z(s;\sigma,\chi)$ admits a meromorphic continuation to the whole complex plane $\C$. The set of the singularities
equals to $\{s_{k}^{\pm}=\pm i\lambda_{k}:\lambda_{k}\in \spec(D^{\sharp}_{\chi}(\sigma)),k\in\N\}$.
The orders of the singularities are equal to $\frac{1}{2}(\pm m_{s}(\lambda_{k})+m(\lambda^{2}_{k}))$.
For $\lambda_{0}=0$, the order of the singularity is equal to $m(0)$.
\end{thm}
\begin{thm}
For every $\sigma\in\widehat{M}$, the Ruelle zeta function $R(s;\sigma,\chi)$ admits a meromorphic continuation to the whole complex plane $\C$.
\end{thm}
\textbf{Acknowledgment.} The author wishes
to acknowledge the contribution of her supervisor Werner M\"{u}ller,
as the present paper is part of the author's PhD thesis. 

\begin{center}
\section{\textnormal{Preliminaries}}
\end{center}

Let $X$ be a compact hyperbolic locally symmetric manifold with universal 
covering the real hyperbolic space $\H^{d}$,
where $d=2n+1$ is an odd integer.
%The hyperbolic space is equipped with a Riemannian metric given by the restriction of the metric
%$dd^{2}=dx_{1}^{2}-dx_{2}^{2}\ldots -dx_{d+1}^{2}$ to $\H^{d}$.
%$\mathrm{SO}^0(d,1)$ acts transitively on $\H^{d}$. The stabilizer of the point $(1,0,\ldots,0)$ is $\mathrm{SO}(d)$, which is a maximal compact subgroup of $\mathrm{SO}^0(d,1)$.
We let $G:=\Spin(d,1)$, $K:=\Spin(d)$ be the universal coverings of $\mathrm{SO}^0(d,1)$ and $\mathrm{SO}(d)$, respectively.
We put  $\widetilde{X}:=G/K$. 
We consider the Lie algebras $\mathfrak{g},\mathfrak{k}$ of $G$ and $K$, respectively. Let $ \mathfrak{g}=\mathfrak{k}\oplus\mathfrak{p}$
be the Cartan decomposition of $\mathfrak{g}$.
Let $\mathfrak{a}$ be a Cartan subalgebra of $\mathfrak{p}$.
Then, there exists a canonical isomorphism $T_{eK}\cong \mathfrak{p}$.
We consider the subgroup $A$ of $G$ with Lie algebra $\mathfrak{a}$. We let $M:=\cent_{K}(A)$ be the centralizer of $A$ in $K$. 
Then, $M=\Spin(d-1)$ or $M=\SO(d-1)$. Let $\mathfrak{m}$ be its Lie algebra
and $\mathfrak{b}$ a Cartan subalgebra of $\mathfrak{m}$. Let $\mathfrak{h}$ be a Cartan subalgebra of $\mathfrak{g}$.
Let $\mathfrak{g}_{\C}$, $\mathfrak{h}_{\C}$, $\mathfrak{m}_{\C} $ be the complexifications of $\mathfrak{g}$, $\mathfrak{h}$
and $\mathfrak{m}$, respectively.
We consider the inner product 
\begin{equation}
 \langle Y_1, Y_2\rangle_{0}:=\frac{1}{2(d-1)}B(Y_1,Y_2),\quad Y_1,Y_2 \in \mathfrak{g},
\end{equation}
induced by the Killing form $B(Y_{1},Y_{2}):=\Tr (\ad(Y_{1})\circ \ad(Y_{2}))$ on $\mathfrak{g}\times \mathfrak{g}$.
The restriction of $\langle\cdot,\cdot\rangle_{0}$ to $\mathfrak{p}$ defines
an inner product on $\mathfrak{p}$ and
hence induces a $G$-invariant riemannian metric on $\widetilde{X}$, which 
has constant curvature $-1$. Then, $\widetilde{X}$, equipped with this metric,
is isometric to $\H^{d}$.
Let $\Gamma\subset G$ be a discrete cocompact torsion free subgroup of $G$. Then, 
$X:=\Gamma\backslash \widetilde{X}$
is a locally symmetric compact 
hyperbolic manifold of dimension $d$.

Let $G=KAN$
be the standard Iwasawa decomposition of $G$.
Let $\Delta^{+}(\mathfrak{g},\mathfrak{a})$ be the set of positive roots of the system $(\mathfrak{g},\mathfrak{a})$.
Then, $\Delta^{+}(\mathfrak{g},\mathfrak{a})=\{\alpha\}$.
Let $H_{\R}\in\mathfrak{a}$ such that $\alpha(H_{\R})=1$. 
We set
\begin{equation}
 A^{+}:=\{\exp(tH_{\R})\colon t\in\R^{+}\},
\end{equation}
\begin{align}
 &\rho:=\frac{1}{2}\sum_{\alpha\in \Delta^+(\mathfrak{g},\mathfrak{a})}\dim(\mathfrak{g}_\alpha)\alpha,\\
&\rho_{\mathfrak{m}}:=\frac{1}{2}\sum_{\alpha\in \Delta^+(\mathfrak{m}_{\C},\mathfrak{b})}\alpha.\\\notag
\end{align}
Let $\widehat{K},\widehat{M}$ be the sets of equivalent classes of 
irreducible unitary representations of $K$ and $M$, respectively. 
Let $\nu_{\tau}$, $\nu_{\sigma}$ be the highest weights of $\tau\in \widehat{K}$ and $\sigma\in \widehat{M}$,
respectively.
Then, by  \cite[p. 20]{BO}
\begin{align*}
  \nu_{\tau}=(\nu_{1},\ldots,\nu_{n}),
\end{align*}
where $\nu_{1}\geq\ldots\geq\nu_{n}$ and $\nu_{i},i=1,\ldots,n$ are all integers or all half integers
and
\begin{align}
  \nu_{\sigma}=(\nu_{1},\ldots,\nu_{n-1},\nu_{n}),
\end{align}
where $\nu_{1}\geq\ldots\geq\nu_{n-1}\geq\lvert\nu_{n}\rvert$ and $\nu_{i},i=1,\ldots,n$ are all integers or all half integers .

Let $(s,S)$ be the spin representation of $K$ and $(s^{\pm},S^{\pm})$ the half spin representations of $M$ as in \cite[p. 7]{Spil}.
By \cite[p. 20]{BO}, 
the highest weight $ \nu_{s}$ of $s$ is given by $=(\frac{1}{2},\ldots,\frac{1}{2})$,
and the highest weights $\nu_{s^{\pm}}$ of $s^{\pm}$ are
$\nu_{s^{\pm}}=(\frac{1}{2},\ldots,\pm\frac{1}{2})$. 
\newpage

\begin{center}
{\section{\textnormal{Twisted Selberg and Ruelle zeta functions}}}
\end{center}
The twisted Ruelle and Selberg zeta functions are associated with the geodesic flow on the
sphere vector bundle $S(X)=\Gamma\backslash G/M$ of $X$. 
By (\cite{GKM}), there is a 1-1 correspondence between the closed geodesics on a manifold $X$ with negative sectional curvature 
and the non-trivial conjugacy classes of $\Gamma$.
Since $\Gamma$ is a cocompact subgroup of $G$, we realize every element $\gamma\in\Gamma-\{e\}$ as hyperbolic.
By \cite[Lemma 6.5]{Wa}, there exist a $g\in G$, a $m_{\gamma} \in M$, and an 
$a_{\gamma} \in A^{+}$, such that $ g^{-1}\gamma g=m_{\gamma}a_{\gamma}$. 
The element $m_{\gamma}$ is determined up to conjugacy classes in $M$, and the element $a_{\gamma}$ depends only on $\gamma$.

Let $c_{\gamma}$ be the closed geodesic on $X$, associated with the conjugacy class $[\gamma]$. 
We denote by $l(\gamma)$ the length of $c_{\gamma}$.
An element $\gamma\in\Gamma$ is primitive, if there exists no $n\in\N$ with $n>1$ and $\gamma_{0}\in\Gamma$ such that $\gamma=\gamma_{0}^{n}$.
The prime geodesics on $X$ are those geodesics whose periodic orbits are of minimal length.
They are associated with the primitive elements $\gamma_{0}\in\Gamma$.

Let now $\chi\colon\Gamma\rightarrow \GL(V_{\chi})$ be an arbitrary finite dimensional representation of $\Gamma$ and $\sigma\in \widehat{M}$.

\begin{defi}
The twisted Selberg zeta function $Z(s;\sigma,\chi)$ for 
$X$ is defined  by the infinite product
\begin{equation}
Z(s;\sigma,\chi):=\prod_{\substack{[\gamma]\neq{e}\\ [\gamma]\prim}} \prod_{k=0}^{\infty}\det\big(\Id-(\chi(\gamma)\otimes\sigma(m_\gamma)\otimes S^k(\Ad(m_\gamma a_\gamma)|_{\overline{\mathfrak{n}}})) e^{-(s+|\rho|)\lvert l(\gamma)}\big),
\end{equation}
where $s\in \C$, $\overline{\mathfrak{n}}=\theta \mathfrak{n}$ is the sum of the negative root spaces of $\mathfrak{a}$,
$S^k(\Ad(m_\gamma a_\gamma)_{\overline{\mathfrak{n}}})$ denotes the $k$-th
symmetric power of the adjoint map $\Ad(m_\gamma a_\gamma)$ restricted to $\mathfrak{\overline{n}}$, and $\rho$ is as in (2.3).
\end{defi}
By \cite[Proposition 3.4]{Spil}, there exists a positive constant $c$, such that the infinite product in (3.1) converges 
absolutely and uniformly on compact subsets of the half-plane Re$(s)>c$.
\begin{defi} 
The twisted Ruelle zeta function $ R(s;\sigma,\chi)$ for $X$ is defined by the infinite product
\begin{equation}
 R(s;\sigma,\chi):=\prod_{\substack{[\gamma]\neq{e}\\ [\gamma]\prim}}\det\big(\Id-\chi(\gamma)\otimes\sigma(m_{\gamma})e^{-sl(\gamma)}\big)^{(-1)^{d-1}}.
\end{equation}
\end{defi}
By \cite[Proposition 3.5]{Spil}, there exists a positive constant $r$, such that the infinite product in (3.2) converges 
absolutely and uniformly on compact subsets of the half-plane Re$(s)>r$.

Let  $M'=\Norm_{K}(A)$ be the normalizer of $A$ in $K$.
We define the restricted Weyl group by $ W_{A}:=M'/M$.
Then, $W_{A}$ has order 2. Let $w\in W_{A}$ be a non-trivial element of $W_{A}$, and 
$m_{w}$ a representative of $w$ in $M'$.
The action of $W_{A}$ on $\widehat{M}$ is defined by
\begin{equation*}
(w\sigma)(m):=\sigma(m_{w}^{-1}mm_{w}),\quad m\in M, \sigma\in\widehat{M}.
\end{equation*}
We have already associated the Selberg and Ruelle zeta functions with irreducible representations $\sigma$ of $M$. These representations are chosen precisely to be
the representations arising from restrictions of representations of $K$.
Let $i^{*}:R(K)\rightarrow R(M)$ be the pullback of the embedding
$i:M\hookrightarrow K$.
%, where $R(K)$, $R(M)$ denote the representation rings over $\Z$ of $K$ and $M$, respectively. 
We will distinguish the following two cases:
\begin{itemize}
 \item  {\bf case (a)}: \textit{$\sigma$ is invariant under the action of the restricted Weyl group $W_{A}$.}
 \item {\bf case (b)}: \textit{$\sigma$ is not invariant under the action of the restricted Weyl group $W_{A}$.}
\end{itemize}
In case (b), we define the following twisted zeta functions.
\begin{defi}
Let $\chi\colon\Gamma\rightarrow \GL(V_{\chi})$ be a finite dimensional representation of $\Gamma$ and $\sigma\in \widehat{M}$.
The symmetrized zeta function $Z(s;\sigma,\chi)$ for 
$X$ is defined  by 
 \begin{equation}
 S(s;\sigma,\chi):=Z(s;\sigma,\chi)Z(ws;\sigma,\chi),
\end{equation}
where $w$ is a non-trivial element of the restricted Weyl group $W_{A}$.
\end{defi}
\begin{defi}
Let $\chi\colon\Gamma\rightarrow \GL(V_{\chi})$ be a finite dimensional representation of $\Gamma$ and $\sigma\in \widehat{M}$.
The super zeta function $Z(s;\sigma,\chi)$ for 
$X$ is defined  by 
 \begin{equation}
 Z^{s}(s;\sigma,\chi):=\frac{Z(s;\sigma,\chi)}{Z(ws;\sigma,\chi)},
\end{equation}
where $w$ is a non-trivial element of the restricted Weyl group $W_{A}$.
\end{defi}
\begin{defi}
Let $\chi\colon\Gamma\rightarrow \GL(V_{\chi})$ be a finite dimensional representation of $\Gamma$ and $\sigma\in \widehat{M}$.
The super Ruelle zeta function $R^{s}(s;\sigma,\chi)$ for 
$X$ is defined  by 
\begin{equation}
 R^{s}(s;\sigma,\chi):=\frac{R(s;\sigma,\chi)}{R(s;w\sigma,\chi)},
\end{equation}
where $w$ is a non-trivial element of the restricted Weyl group $W_{A}$.
\end{defi}
We compute here the logarithmic derivative of the symmetrized and super zeta function.
\begin{lem}
Let 
\begin{equation}
 L_{sym}(\gamma;\sigma):=\frac{\tr(\chi(\gamma)\otimes\sigma(m_{\gamma}))e^{-|\rho|l(\gamma)}}
 {\det{(\Id-\Ad(m_{\gamma}a_{\gamma})_{\overline{n}})}}.
\end{equation}
Then we have 
\begin{enumerate}
\item  
 The logarithmic derivative of the symmetrized zeta function $S(s;\sigma,\chi)$ is given by
 \begin{equation}\label{f:log der symmetrized}
 L_{S}(s):=\frac{d}{ds}\log(S(s;\sigma,\chi))=\sum_{[\gamma]\neq{e}}\frac{l(\gamma)}{n_{\Gamma}(\gamma)}L_{sym}(\gamma;\sigma+w\sigma)
e^{-sl(\gamma)}.
\end{equation} 
\item 
 The logarithmic derivative of the super zeta function $Z^{s}(s;\sigma,\chi)$ is given by
 \begin{equation}\label{f:log der super}
 L^{s}(s):=\frac{d}{ds}\log(Z^{s}(s;\sigma,\chi))=\sum_{[\gamma]\neq{e}}\frac{l(\gamma)}{n_{\Gamma}(\gamma)}L_{sym}(\gamma;\sigma-w\sigma)
e^{-sl(\gamma)}.
\end{equation}
\end{enumerate}
\end{lem}

\begin{proof}
\begin{enumerate}
 \item In case (b), for the symmetrized zeta function $S(s;\sigma,\chi)$, we see by equation (3.3)
\begin{align*}
\frac{d}{ds}\log(S(s;\sigma,\chi))&=\frac{d}{ds}\log(Z(s;\sigma,\chi)+\frac{d}{ds}\log(Z(s;w\sigma,\chi)\\
 &=\sum_{[\gamma]\neq{e}}\frac{l(\gamma)}{n_{\Gamma}(\gamma)}\tr(\chi(\gamma)\otimes\sigma(m_\gamma))\frac{e^{-sl(\gamma)}e^{-|\rho|l(\gamma)}}{\det(1-\Ad(m_\gamma a_\gamma)_{\overline{\mathfrak{n}}})}\\
 &+\sum_{[\gamma]\neq{e}}\frac{l(\gamma)}{n_{\Gamma}(\gamma)}\tr(\chi(\gamma)\otimes w\sigma(m_\gamma))\frac{e^{-sl(\gamma)}e^{-|\rho|l(\gamma)}}{\det(1-\Ad(m_\gamma a_\gamma)_{\overline{\mathfrak{n}}})}\\
 &=\sum_{[\gamma]\neq{e}}\frac{l(\gamma)}{n_{\Gamma}(\gamma)}L_{sym}(\gamma;\sigma+w\sigma)e^{-sl(\gamma)}.
\end{align*}

\item In case (b), for the super zeta function $Z^{s}(s;\sigma,\chi)$, we see by equation (3.4)
\begin{align*}
\frac{d}{ds}\log(Z^{s}(s;\sigma,\chi))&=\frac{d}{ds}\log(Z(s;\sigma,\chi)-\frac{d}{ds}\log(Z(s;w\sigma,\chi)\\
 &=\sum_{[\gamma]\neq{e}}\frac{l(\gamma)}{n_{\Gamma}(\gamma)}\tr(\chi(\gamma)\otimes\sigma(m_\gamma))\frac{e^{-sl(\gamma)}e^{|\rho|l(\gamma)}}{\det(1-\Ad(m_\gamma a_\gamma)_{\overline{\mathfrak{n}}})}\\
 &-\sum_{[\gamma]\neq{e}}\frac{l(\gamma)}{n_{\Gamma}(\gamma)}\tr(\chi(\gamma)\otimes w\sigma(m_\gamma))\frac{e^{-sl(\gamma)}e^{|\rho|l(\gamma)}}{\det(1-\Ad(m_\gamma a_\gamma)_{\overline{\mathfrak{n}}})}\\
 &=\sum_{[\gamma]\neq{e}}\frac{l(\gamma)}{n_{\Gamma}(\gamma)}L_{sym}(\gamma;\sigma-w\sigma)e^{-sl(\gamma)}.
\end{align*}
\end{enumerate}
\end{proof}

\begin{center}
\section{\textmd{The twisted Dirac operator}}
\end{center}

Let $\sigma\in\widehat{M}$ be an irreducible representation of $M$ with highest weight $\nu_{\sigma}$ as in (2.5).
We recall that $\nu_{n}$ denotes the last coordinate of  $\nu_{\sigma}$.
Let $s$ be the spin representation of $K$. Since $d-1$ is an even integer, $s$ splits into two irreducible half-spin representations $(s^{+},S^{+})$, $(s^{-},S^{-})$ of $M$.
Let $\Cl(\mathfrak{p})$ be the Clifford algebra of $\mathfrak{p}$ with respect to the inner product $\langle\cdot,\cdot\rangle_{0}$, as in (2.1), restricted to $\mathfrak{p}$.
Let $
 \cdot\colon\mathfrak{p}\otimes S\rightarrow S
$
be the Clifford multiplication on $\mathfrak{p}\otimes S$.
Let $H$ be in $\mathfrak{a}^{+}$, where $\mathfrak{a}^{+}$ is the Lie algebra of $A^{+}$ and $A^{+}$ is defined as in (2.2). Since $M$ centralizes $\mathfrak{a}$,
the Clifford multiplication by $H$ preserves the decomposition $S=S^{+}\oplus S^{-}$. The Clifford multiplication by $H^{2}$ acts as $-\Id$. Then, $H$
acts on $S^{\pm}$ with eigenvalues $\pm i$. Hence, we can consider the Clifford multiplication by $H$ 
as multiplication by $\pm i \sign(\nu_{n})$.

We consider the connection $\nabla$ in $\Cl(\mathfrak{p})$, induced by the canonical connection in the tangent frame bundle of $X$.
Let $L$ be any bundle of left modules over $\Cl(\mathfrak{p})$ over $\widetilde{X}$, i.e., a spinor bundle over $\widetilde{X}$. We lift the connection $\nabla$ in $L$ and obtain a connection also denoted by $\nabla$.
The Dirac operator $D\colon C^{\infty}(X,L)\rightarrow C^{\infty}(X,L)$  is defined as
\begin{equation*}
 D:C^{\infty}(X,L)\overset{\nabla}\rightarrow
 C^{\infty}(X,T^{*}X\otimes L)\overset{g}\rightarrow C^{\infty}(X,TX\otimes L)
\overset{\cdot}\rightarrow C^{\infty}(X,L),
\end{equation*}
where we identify $TX\cong T^{*}X$ using the riemannian metric, and 
$\cdot$ denotes the Clifford multiplication as above.
Locally, it can be described as
\begin{equation*}
 Df\equiv\sum_{i=1}^{d}e_{i}\cdotp\nabla_{e_{i}}f,
\end{equation*}
where $(e_{1},\ldots,e_{d})$ is a local orthonormal frame for $T_{x}X, x\in X$.
The bundle $L$ is a Dirac bundle over $\widetilde{X}$. This means that 
\begin{itemize}
 \item the Clifford multiplication by unit vectors in $\Cl(\mathfrak{p})$ is orthogonal i.e.,
\begin{equation*}
 \langle e f_{1},e f_{2}\rangle=\langle f_{1},f_{2}\rangle,
\end{equation*}
for all unit vectors $e\in T_{x}\widetilde{X}, x\in\widetilde{X}$ and all $f_{1},f_{2}\in L_{x}$, where $L_{x}$ denotes the fiber of $L$ over $x\in \widetilde{X}$.
\item the connection $\nabla$ satisfies the product rule
\begin{equation*}
 \nabla(\phi f)=(\nabla\phi)\cdotp f+ \phi\cdotp(\nabla f),
\end{equation*}
for all $\phi\in C^\infty(X,\Cl(\mathfrak{p}))$ and all $f\in C^{\infty}(X,L)$.
\end{itemize}
The operator $D$ is an elliptic (\cite[Lemma 5.1]{LM}), formally self-adjoint (\cite[Proposition 5.3]{LM}) operator of first order.
We want to define twisted Dirac operators acting on smooth sections of vector bundles associated with the representations $\sigma$ of $M$
and arbitrary represenations $\chi$ of $\Gamma$.
\begin{prop}
Let $\sigma\in\widehat{M}$. Then, there exists a unique element  $\tau(\sigma)\in \widehat{K}$ and a splitting 
\begin{equation*}
 s \otimes \tau(\sigma)=\tau^{+}(\sigma)\oplus \tau^{-}(\sigma)
\end{equation*}
where $\tau^{+}(\sigma),\tau^{-}(\sigma)\in R(K)$
such that 
\begin{equation}
   \sigma+w\sigma=i^{*}(\tau^{+}(\sigma)-\tau^{-}(\sigma))
\end{equation}
%\begin{equation}
%  \sigma-w\sigma=\sign(\nu_{n})(s^{+}-s^{-})i^{*}(\tau(\sigma))
%\end{equation}
\end{prop}
\begin{proof}
 This is proved in \cite[Proposition 1.1, (3)]{BO}.
\end{proof}
We define the representation $\tau_{s}(\sigma)$ of $K$ by
\begin{equation}
 \tau_{s}(\sigma):= s\otimes\tau(\sigma),
\end{equation}
with representation space $V_{\tau_{s}(\sigma)}=S\otimes V_{\tau(\sigma)}$,
where $V_{\tau(\sigma)}$ is the representation space of $\tau(\sigma)$.\\
We consider the homogeneous vector bundle $\widetilde{E}_{\tau(\sigma)}$ over $\widetilde{X}$ given by
\begin{equation*}
 \widetilde{E}_{\tau(\sigma)}=G\times_{\tau(\sigma)}V_{\tau(\sigma)}\rightarrow\widetilde{X}.
\end{equation*}
The vector bundle $\widetilde{E}_{\tau_{s}(\sigma)}:= \widetilde{E}_{\tau(\sigma)}\otimes S$ over $\widetilde{X}$ carries a connection $\nabla^{\tau_{s}(\sigma)}$, defined by the formula 
%a fiber metric $h^{\tau_{s}(\sigma)}$ 
\begin{equation*}
 \nabla^{\tau_{s}(\sigma)}=\nabla^{\tau(\sigma)}\otimes 1+1\otimes\nabla.
\end{equation*}
where $\nabla^{\tau(\sigma)}$ denotes the canonical connection in  $\widetilde{E}_{\tau(\sigma)}$.
We extend the Clifford multiplication by requiring that it acts on $V_{\tau_{s}(\sigma)}=S\otimes V_{\tau(\sigma)}$ as follows.
\begin{equation*}
  e\cdot (\phi\otimes \psi)=(e\cdot \phi)\otimes \psi, \quad e\in\Cl(\mathfrak{p}), \phi\in S, \psi\in V_{\tau(\sigma)}.
\end{equation*}

%\begin{equation*}
% h^{\tau_{s}(\sigma)}(ef_{1},ef_{2})= h^{\tau_{s}(\sigma)}(f_{1},f_{2}) 
%\end{equation*}
%for all $e\in \mathfrak{p}$ with $\lvert e \rvert=1$, and for all $f_{1},f_{2}\in \widetilde{E}_{\tau_{s}(\sigma)}$,
%and also the connection $\nabla^{\tau_{s}(\sigma)}$ satisfies the product law
%\begin{equation*}
%\nabla^{\tau_{s}(\sigma)}(Y f)= \nabla^{\tau_{s}(\sigma)}Y \cdot f+Y \cdot \nabla^{\tau_{s}(\sigma)}f
%\end{equation*}
%for all $Y\in C^{\infty}(\widetilde{X},\mathfrak{p})$ and for all $f\in C^{\infty}(\widetilde{X},\widetilde{E}_{\tau_{s}(\sigma)})$.
%This follows from \cite[Proposition 5.10]{LM}.

We define the Dirac operator $\widetilde{D}(\sigma)$ acting on $C^{\infty}(\widetilde{X},V_{\tau_{s}(\sigma)})$ by
\begin{equation*}
 \widetilde{D} (\sigma)f=\sum_{i=1}^{d}e_{i}\cdot \nabla_{e_{i}}^{\tau_{s}(\sigma)}f,
\end{equation*}
where $(e_{1},\ldots,e_{d})$ is local orthonormal frame for $T_{x}{\widetilde{X}}$ and $f\in C^{\infty}(\widetilde{X},V_{\tau_{s}(\sigma)})$.
The space of smooth sections $C^{\infty}(\widetilde{X},V_{\tau_{s}(\sigma)})$ can be identified with $C^{\infty}(G;\tau_{s}(\sigma))$ as in \cite[equation (5.1)]{Spil}.\\
%The Dirac operator $\widetilde{D}$ is an elliptic formally self-adjoint operator of first order.

Let now $\chi:\Gamma\rightarrow \GL(V_{\chi})$ be an arbitrary finite dimensional representation of $\Gamma$. Let $E_{\chi}$
be the associated flat vector bundle over $X$.
Let $E_{\tau_{s}(\sigma)}:= \Gamma\backslash \widetilde{E}_{\tau_{s}(\sigma)}$ be the locally homogeneous vector bundle over $X$. 
% Let $\widetilde{E}_{\chi}$ be the pullback of $E_{\chi}$ to $\widetilde{X}$.
We consider the product vector bundle $E_{\tau_{s}(\sigma)}\otimes E_{\chi}$ over $X$
%$E_{\tau,\chi}=\Gamma\backslash(\widetilde{E}_{\tau_{s}(\sigma)}\otimes \widetilde{E}_{\chi})$ over $X$,
and we equip this bundle with the product connection $\nabla^{E_{\tau_{s}(\sigma)}\otimes E_{\chi}}$ defined by
\begin{equation*}
 \nabla^{E_{\tau_{s}(\sigma)}\otimes E_{\chi}}=\nabla^{E_{\tau_{s}(\sigma)}}\otimes 1+1\otimes\nabla^{E_{\chi}}.
\end{equation*}
We consider the Clifford multiplication on $(V_{\tau_{s}(\sigma)}\otimes V_{\chi})$ by requiring that it acts only on $V_{\tau_{s}(\sigma)}$, i.e.,
\begin{equation*}
 e\cdot (w\otimes v)=(e\cdot w)\otimes v, \quad e \in\Cl(\mathfrak{p}), w\in V_{\tau_{s}(\sigma)}, v\in V_{\chi}.
\end{equation*}

For our proposal, we introduce the twisted Dirac operator $D^{\sharp}_{\chi}(\sigma)$ associated with $ \nabla^{E_{\tau_{s}(\sigma)}\otimes E_{\chi}}$.
We want to describe it locally. We consider an open subset of $X$ such that $E_{\chi}\lvert_{U}$ is trivial. Let $(v_{j}), j=1,\dots,m$,  
be a basis of flat sections of $E_{\chi}\lvert_{U}$, where $m=\rank(E_\chi)$, and $\phi_{j}\in C^{\infty}(U, E_{\tau_{s}(\sigma)}\lvert_{U})$.
Then,
\begin{equation*}
 E_{\tau_{s}(\sigma)}\otimes E_{\chi}\lvert_{U}\cong\bigoplus_{j=1}^{m}E_{\tau_{s}(\sigma)}\lvert_{U},
\end{equation*}
and for each $\phi \in C^{\infty}(U,E_{\tau_{s}(\sigma)} \otimes E_{\chi}\lvert_{U})$,
\begin{equation*}
\phi=\sum_{j=1}^{m}\phi_{j} \otimes v_{j}.
\end{equation*}
The product connection is given by
\begin{equation*}
  \nabla^{E_{\tau_{s}(\sigma)}\otimes E_{\chi}}(\phi)=\sum_{j=1}^{m}(\nabla^{E_{\tau_{s}(\sigma)}})(\phi_{j})\otimes v_{j}.
\end{equation*}
Then the Dirac operator is described as follows.
\begin{align}
D^{\sharp}_{\chi}(\sigma)\phi\notag&=\sum_{i=1}^{d}e_{i}\cdot\nabla_{e_{i}}^{E_{\tau_{s}(\sigma)}\otimes E_{\chi}}(\phi)\\\notag
&=\sum_{i=1}^{d}e_{i}\cdot\big(\sum_{j=1}^{m}(\nabla_{e_{i}}^{{E}_{\tau_{s}(\sigma)}})(\phi_{j})\otimes v_{j}\big)\\
&=\sum_{i=1}^{d} \sum_{j=1}^{m}e_{i}\cdot\big((\nabla_{e_{i}}^{{E}_{\tau_{s}(\sigma)}})(\phi_{j})\otimes v_{j}\big).
%&=\sum_{i=1}^{d} \sum_{j=1}^{m}e_{i}\cdot(\nabla_{e_{i}}^{{E}_{\tau_{s}(\sigma)}})(\phi_{j})\otimes v_{j}.
\end{align}
We consider the pullbacks $\widetilde{E}_{\tau_{s}(\sigma)}, \widetilde{E}_{\chi}$ to $\widetilde{X}$ of $E_{\tau_{s}(\sigma)},E_{\chi}$, respectively,
then, $\widetilde{E}_{\chi}\cong \widetilde{X}\times V_{\chi}$.
We have 
\begin{equation*}
 C(\widetilde{X}, \widetilde{E}_{\tau_{s}(\sigma)}\otimes \widetilde{E}_{\chi})\cong  C(\widetilde{X}, \widetilde{E}_{\tau_{s}(\sigma)})\otimes V_{\chi}.
\end{equation*}
With respect to this isomorphism, it follows from (4.3) that the lift ${\widetilde{D}}^{\sharp}_{\chi}(\sigma)$ of the twisted Dirac operator $D^{\sharp}_{\chi}(\sigma)$ 
to $\widetilde{X}$ is of the form
\begin{equation}
\widetilde{D}^{\sharp}_{\chi}(\sigma)=\widetilde{D}(\sigma)\otimes \Id_{V_{\chi}}.
\end{equation}
We recall the definition of the operator $A_{\chi}^{\sharp}(\sigma)$ acting on $C^{\infty}(X,E(\sigma)\otimes E_{\chi})$ from \cite[equation 5.26]{Spil}
\begin{equation*}
 A_{\chi}^{\sharp}(\sigma):=\bigoplus_{m_{\tau}(\sigma)\neq 0}A^{\sharp}_{\tau,\chi}+c(\sigma),
\end{equation*}
where 
\begin{equation*}
 c(\sigma):=-\lvert \rho \rvert^{2}-\lvert\rho_{m}\rvert^{2}+\lvert \nu_{\sigma}+\rho_{m}\rvert^{2},
\end{equation*}
%where $\nu_{\sigma}$ is the highest weight of $\sigma\in\widehat{M}$ as in (2.5)
and $\rho,\rho_{m}$ are defined by (2.3) and (2.4), respectively.
The operator $A_{\tau,\chi}^{\sharp}$ 
is induced by the twisted Bochner-Laplace operator 
$\Delta^{\sharp}_{\tau,\chi}$ and
acts on smooth sections of the twisted vector bundle $E(\sigma)\otimes E_{\chi}$ over $X$ (see \cite[p. 27-28]{Spil}).
We recall here the definition of the vector bundle $E(\sigma)$.
We consider always representations $\sigma$ of $M$ coming from restricions of representations of $K$.
Let $\tau_{\sigma}\in R(K)$ with $\tau_{\sigma}:=\tau^{+}(\sigma)-\tau^{-}(\sigma)$.
By \cite[Proposition 1.1]{BO}, there exists unique integers $m_{\tau}(\sigma)\in\{-1,0,1\}$, which are equal to zero except for finitely many $\tau\in \widehat{K}$,
such that for the case (b) (see p. 8)
\begin{equation}
 \sigma+w\sigma=\sum_{\tau\in\widehat{K}}m_{\tau}(\sigma)i^{*}(\tau).
\end{equation}
Then, the locally homogeneous vector bundle $E(\sigma)$ associated with $\tau$ is of the form
\begin{equation*}
 E(\sigma)=\bigoplus_{\substack{\tau\in\widehat{K}\\m_{\tau}(\sigma)\neq 0}}E_{\tau},
\end{equation*}
where $E_{\tau}$ is the locally homogeneous vector bundle associated with $\tau\in\widehat{K}$.
Therefore, the vector bundle $E(\sigma)$ has a grading $E(\sigma)=E(\sigma)^{+}\oplus E(\sigma)^{-}$.
This grading is defined exactly by the positive or negative sign of $m_{\tau}(\sigma)$.
Let $\widetilde{E}(\sigma)$ be the pullback of $E(\sigma)$ to $\widetilde{X}$.
Then,
\begin{equation*}
 \widetilde{E}(\sigma)=\bigoplus_{\substack{\tau\in\widehat{K}\\m_{\tau}(\sigma)\neq0}}\widetilde{E}_{\tau}.
\end{equation*}
One can easily observe that by Proposition 4.1 (see also \cite[Proposition 5.1 and p. 27-28]{BO}), 
that up to a $\Z_{2}$ grading, one can identify the two vector bundles $E(\sigma)$ and $E_{\tau_{s}(\sigma)}$.

We consider the lift $\widetilde{A}_{\chi}^{\sharp}(\sigma)$ of $A_{\chi}^{\sharp}(\sigma)$ to the universal covering $\widetilde{X}$.
Then,
\begin{align}
 \widetilde{A}_{\chi}^{\sharp}(\sigma)\notag&=\bigoplus_{m_{\tau}(\sigma)\neq 0}\widetilde{A}^{\sharp}_{\tau,\chi}+c(\sigma)\\\notag
&=\bigoplus_{m_{\tau}(\sigma)\neq 0}(\widetilde{A}_{\tau}\otimes \Id_{V_{\chi}})+c(\sigma)\\
&=\bigoplus_{m_{\tau}(\sigma)\neq 0}\big(\widetilde{A}_{\tau}+c(\sigma)\big)\otimes \Id_{V_{\chi}}.
\end{align}
The Parthasarathy formula from \cite[equation (1.11)]{BO} states
\begin{equation}
(\widetilde{D}(\sigma))^{2}=\bigoplus_{m_{\tau}(\sigma)\neq 0}\big(\widetilde{A}_{\tau}+c(\sigma)\big)
\end{equation}
If we combine (4.4), (4.6) and (4.7) the Parthasarathy formula generalizes as
\begin{equation}
(D^{\sharp}_{\chi}(\sigma))^{2}=A_{\chi}^{\sharp}(\sigma).
\end{equation}

\begin{center}
{\section{\textnormal{Trace formulas}}}
\end{center}
By (4.4), we get 
\begin{equation}
  (\widetilde{D}^{\sharp}_{\chi}(\sigma))^{2}=(\widetilde{D}(\sigma))^{2}\otimes \Id_{V_{\chi}}
\end{equation}
The square of the twisted Dirac operator $(D^{\sharp}_{\chi}(\sigma))^{2}$ acting on smooth sections of $E_{\tau_{s}(\sigma)}\otimes E_{\chi}$ is not a self-adjoint operator in general.
Nevertheless, its principal symbol is given by
\begin{equation*}
 \sigma_{(D^{\sharp}_{\chi}(\sigma))^{2}}(x,\xi)=\lVert\xi\rVert^{2}\otimes\Id_{(V_{\tau_{s}(\sigma)}\otimes V_{\chi})_{x}}, \quad x\in X,\quad \xi\in T_{x}^{*}X,\xi\neq 0.
 \end{equation*}
Therefore, $({D^{\sharp}}_{\chi}(\sigma))^{2}$ is a second order elliptic differential operator with nice spectral properties, i.e., its spectrum is discrete
and contained in a translate of a positive cone $C\subset\C$.
Furthermore, we can define the integral operators $e^{-t({D^{\sharp}}_{\chi}(\sigma))^{2}}$ and $D^{\sharp}_{\chi}(\sigma)e^{-t({D^{\sharp}}_{\chi}(\sigma))^{2}}$, and derive a 
corresponding trace formula.
% such that $\R^{+}\subset C$. \\
%Using the calculus of pseudodifferential operators (App. ) we define the operator
%$\widetilde{D}^{\sharp}_{\chi}(\sigma)e^{-t{(\widetilde{D}^{\sharp}_{\chi}(\sigma)})^{2}}$
We introduce here a more general setting.

\begin{setting}
Let $E\rightarrow X$ be a complex vector bundle over a smooth compact riemannian manifold $X$ of dimension $d$.
Let $D:C^{\infty}(X,E)\rightarrow C^{\infty}(X,E)$ be an elliptic differential operator of order $m\geq 1$.
Let $\sigma_{D}$ be its principal symbol.
\end{setting}
\begin{defi}
 A spectral cut\index{spectral cut} is a ray
\begin{equation*}
 R_{\theta}:=\{\rho e^ {i\theta}: \rho\in[0,\infty]\},
\end{equation*}
where $\theta\in[0,2\pi)$.
\end{defi}
%We have to choose a specific angle $\theta\in[0,2\pi)$.
\begin{defi}
 The angle $\theta$ is a principal angle for an elliptic operator $D$ if 
\begin{equation*}
 \spec(\sigma_{D}(x,\xi))\cap R_{\theta}=\emptyset,\quad \forall x\in X,\forall\xi\in T_{x}^{*}X,\xi\neq 0.
\end{equation*}
\end{defi}
\begin{defi}
 We define the solid angle $L_{I}$ associated with a closed interval $I$ of $\R$ by
\begin{equation*}
 L_{I}:=\{\rho e^ {i\theta}: \rho\in(0,\infty), \theta\in I \}.
\end{equation*}
\end{defi}
\begin{defi}
The angle $\theta$ is an Agmon angle\index{Agmon angle} for an elliptic operator $D$, if it is a principal angle for $D$ 
and there exists $\varepsilon>0$ such that 
\begin{equation*}
 \spec(D)\cap L_{[\theta-\varepsilon,\theta+\varepsilon]}=\emptyset.
\end{equation*}
\end{defi}
\begin{lem}
Let $\varepsilon\in(0,\frac{\pi}{2})$ be an angle such that the principal symbol $\sigma_{D}(x,\xi)$ of $D$,
for $\xi\in T_{x}^{*}X,\xi\neq 0$ does not take values in $ L_{[-\varepsilon,\varepsilon]}$.
Then, the spectrum $\spec(D)$ of the operator $D$ 
is discrete and for every $\varepsilon\in(0,\frac{\pi}{2})$ there exist $R>0$ such that $\spec(D)$
is contained in the set $B(0,R)\cup L_{[-\varepsilon,\varepsilon]}\subset \C$.
\end{lem}
\begin{proof}
 The discreteness of the spectrum follows from \cite[Theorem 8.4]{Sh}. For the second statement see \cite[Theorem 9.3]{Sh}.
\end{proof}
Let $\lambda_{k}$ be an eigenvalue of $D$ and 
$V_{\lambda_{k}}$ be the corresponding eigenspace.
This is a finite dimensional subspace of $C^{\infty}(X,E)$ invariant under $D$.
We have that for every $k\in\N$, there exist $N_{k}\in\N$ such that 
\begin{align*}
 &(D-\lambda_{k}\Id)^{N_{k}}V_{\lambda_{k}}=0\\
&\lim_{k\rightarrow \infty}\lvert \lambda_{k}\rvert=\infty.
\end{align*}
\begin{defi}
We call algebraic multiplicity $m(\lambda_{k})$ of the eigenvalue $\lambda_{k}$ the dimension of the corresponding
eigenspace $V_{\lambda_{k}}$.
\end{defi}
By Lemma 5.6, we get Lemma 5.8, which describes the spectrum of the square root $(D^{\sharp}_{\chi}(\sigma))^{2}$ of the twisted Dirac operator.
\begin{lem}
Let $\varepsilon\in(0,\frac{\pi}{2})$ be an angle such that the principal symbol $\sigma_{(D^{\sharp}_{\chi}(\sigma))^{2}}(x,\xi)$ of $(D^{\sharp}_{\chi}(\sigma))^{2}$,
for $\xi\in T_{x}^{*}X,\xi\neq 0$ does not take values in $ L_{[-\varepsilon,\varepsilon]}$.
Then, the spectrum $\spec((D^{\sharp}_{\chi}(\sigma))^{2})$ of the twisted Dirac operator $(D^{\sharp}_{\chi}(\sigma))^{2}$ 
is discrete and for every $\varepsilon$ there exists $R>0$ such that $\spec((D^{\sharp}_{\chi}(\sigma))^{2})$
is contained in the set $B(-1,R)\cup L_{[-\varepsilon,\varepsilon]}\subset \C$.
\end{lem}
\begin{proof}
As in the proof of Lemma 5.6.
\end{proof}
%By \cite{Mk} the space $L^{2}(X,E)$ can be decomposed as
%\begin{equation*}
% L^{2}(X,E)=\overline{\bigoplus_{k\geq 1}V_{\lambda_{k}}}.
%\end{equation*}
%This is the generalization of the eigenspace decomposition of a self-adjoint operator. 
%We note here that, in general, the above decomposition is not a sum of mutually orthogonal subspaces.
Let $\theta$ be an Agmon angle for the operator $(D^{\sharp}_{\chi}(\sigma))^{2}$. Then, by definition of the Agmon angle and Lemma 5.8, there exists $\varepsilon>0$
such that 
\begin{equation*}
 \spec((D^{\sharp}_{\chi}(\sigma))^{2})\cap L_{[\theta-\varepsilon,\theta+\varepsilon]}=\emptyset.
\end{equation*}
 Since $({D^{\sharp}}_{\chi}(\sigma))^{2}$ has discrete spectrum, there exists also an $r_{0}>0$ such that 
\begin{equation*}
\spec((D^{\sharp}_{\chi}(\sigma))^{2})\cap\{z\in\C:\lvert z+1\rvert\leq 2r_{0}\}=\emptyset.
\end{equation*}
 We define a contour $\Gamma_{\theta,r_{0}}$ as follows.
\begin{equation*}
 \Gamma_{\theta,r _{0}}=\Gamma_{1}\cup\ \Gamma_{2}\cup\Gamma_{3},
\end{equation*}
where $\Gamma_{1}=\{-1+re^ {i\theta}\colon\infty>r\geq r_{0}\}$, $\Gamma_{2}=\{-1+r_{0}e^ {ia}\colon\theta\leq a\leq \theta+2\pi\}$,
$\Gamma_{3}=\{-1+re^{i(\theta+2\pi)}\colon r_{0}\leq r< \infty\}$.
On $\Gamma_{1}$, $r$ runs from $\infty$ to $r_{0}$, $\Gamma_{2}$ is oriented counterclockwise, and on $\Gamma_{3}$,
$r$ runs from $r_{0}$ to $\infty$.
We put
\begin{align}
e^{-t(D^{\sharp}_{\chi}(\sigma))^{2}}=&\frac{i}{2\pi}\int_{\Gamma_{\theta,r_{0}}}e^{-t\lambda}\big((D^{\sharp}_{\chi}(\sigma))^{2}-\lambda\Id\big)^{-1} d\lambda\\
D^{\sharp}_{\chi}(\sigma)e^{-t(({D^{\sharp}}_{\chi}(\sigma))^{2}}=&\frac{i}{2\pi}\int_{\Gamma_{\theta,r_{0}}}\lambda^{1/2}e^{-t\lambda}\big((D^{\sharp}_{\chi}(\sigma))^{2}-\lambda\Id\big)^{-1} d\lambda
\end{align}
We have $\lvert e^{-t\lambda}\rvert\leq e^{-t\text{\Re}(\lambda)}$. Furthermore, by \cite[Corollary 9.2]{Sh},
there exist a positive constant $c>0$ such that $\lVert((D^{\sharp}_{\chi}(\sigma))^{2}-\lambda\Id)^{-1}\rVert\leq c \lvert \lambda\rvert^{-1}$.
Hence, the integrals in (5.2) and (5.3) are  well defined.

%We also recall that $e^{-t\widetilde{A}_{\tau,\chi}}$ is an integral operator with smooth kernel, given by
%\begin{equation*}
%  e^{-t\widetilde{A}_{\tau,\chi}}f(g)=\int_{G}H^{\tau}_{t}(g^{-1}g')f(g')dg', 
%\end{equation*}
%where $f\in L^{2}(\widetilde{X}, \widetilde{E}_{\tau}\otimes \widetilde{E}_{\chi})$, $g\in G$.
%The kernel $H^{\tau}_{t}$   belongs to the Harish-Chandra $L^{q}$-Schwartz space, $\mathfrak{C}^{q}(G)\otimes \End(V_{\tau})$, for every $q>0$.\\
By \cite[Lemma 2.4]{M1}, $D^{\sharp}_{\chi}(\sigma)e^{-t{(D^{\sharp}_{\chi}(\sigma)})^{2}}$
is an integral operator. Let $ K_{t}^{\tau_{s}(\sigma),\chi}$ be its kernel function.
Let $F$ be a fundamental domain of $\Gamma$. We consider the space $L^{2}(\widetilde{X},\widetilde{E}_{\tau_{s}(s)}\otimes \widetilde{E}_{\chi})^{\Gamma}$
of sections $f$ of $\widetilde{E}_{\tau_{s}(s)}\otimes \widetilde{E}_{\chi}$ such that $f(\gamma \widetilde{x})=\gamma f(\widetilde{x})$, $\forall\gamma\in\Gamma, \widetilde{x}\in\widetilde{X}$.\\
For $f\in L^{2}(X,E_{\tau_{s}(s)}\otimes E_{\chi})\cong L^{2}(\widetilde{X},\widetilde{E}_{\tau_{s}(s)}\otimes \widetilde{E}_{\chi})^{\Gamma}$, we have
\begin{align}
D^{\sharp}_{\chi}(\sigma)e^{-t{(D^{\sharp}_{\chi}(\sigma)})^{2}}f(x)\notag&=\int_{X}K_{t}^{\tau_{s}(\sigma),\chi}(x,y)f(y)dy\\\notag
&=\int_{\widetilde{X}}(K_{t}^{\tau_{s}(\sigma)}(\widetilde{x},\widetilde{y})\otimes \Id_{V_{\chi}})f(\widetilde{y})d\widetilde{y}\\
&=\sum_{\gamma \in \Gamma}\int_{F}
(K_{t}^{\tau_{s}(\sigma)}(\widetilde{x},\gamma\widetilde{y})\otimes \chi(\gamma)\Id_{V_{\chi}})f(\widetilde{y})d\widetilde{y}, 
\end{align}
%where $\widetilde{D}^{\sharp}_{\chi}(\sigma)$ denotes the pullback of the operator $D^{\sharp}_{\chi}(\sigma)$ to $\widetilde{X}$, $f\in L^{2}(\widetilde{X}, \widetilde{E}_{\tau}\otimes \widetilde{E}_{\chi}), g,g'\in G$.
%Since the the kernel $H^{\tau_{s}(\sigma)}_{t}$ is induced by the kernel corresponding to the integral operator $e^{-t{(\widetilde{D}^{\sharp}_{\chi}(\sigma)})^{2}}$,
%it belongs to the Harish-Chandra's $L^{q}$-Schwartz space,$(\mathfrak{C}^{q}(G)\otimes \End(V_{\tau_{s}(\sigma)}))$, for every $q>0$,
%Hence, $ \widetilde{D}^{\sharp}_{\chi}(\sigma)e^{-t{(\widetilde{D}(\sigma)}^{\sharp})^{2}}$ is ans integral operator with smooth kernel, and of trace-class and therefore we can produce a trace formula for it. \\
where $x,y\in X$ and $\widetilde{x},\widetilde{y}\in\widetilde{X}$ are lifts of $x,y$ to $\widetilde{X}$, respectively.
The kernel function $K_{t}^{\tau_{s}(\sigma)}$ is the kernel 
associated with the operator $\widetilde{D}(\sigma)e^{-t(\widetilde{D}(\sigma))^{2}}$. It belongs to the Harish-Chandra $L^{q}$-Schwartz space$
(\mathcal{C}^{q}(G)\otimes \End(V_{\tau_{s}(\sigma)}))^{K\times K}$,
as it is defined in \cite[p. 161-162]{BM}.
Hence, we can interchange summation and integration in the right hand side of (5.4) and get
\begin{equation*}
 K_{t}^{\tau_{s}(\sigma),\chi}(x,x')=\sum_{\gamma \in \Gamma}K_{t}^{\tau_{s}(\sigma)}(g^{-1}\gamma g')\otimes \chi(\gamma),
\end{equation*}
where $x=\Gamma g, x'=\Gamma g',g,g'\in G$.\\
By \cite[Proposition 2.5]{M1}, $D^{\sharp}_{\chi}(\sigma)e^{-t{(D^{\sharp}_{\chi}(\sigma)})^{2}}$ is a trace class operator, and its trace is given by
\begin{equation*}
\Tr(D^{\sharp}_{\chi}(\sigma)e^{-t(D^{\sharp}_{\chi}(\sigma))^{2}})=\sum_{\gamma \in \Gamma}\tr \chi(\gamma)\int_{\Gamma\backslash G}\tr K_{t}^{\tau_{s}(\sigma)}(g^{-1}\gamma g)dg.
\end{equation*}
We put
\begin{equation}
 k_{t}^{\tau_{s}(\sigma)}(g)=\tr K_{t}^{\tau_{s}(\sigma)}(g).
\end{equation} 
We use the trace formula in \cite[Proposition 6.1]{M1} for non-unitary twists.
 \begin{align*}
  \Tr(D^{\sharp}_{\chi}(\sigma)e^{-t(D^{\sharp}_{\chi}(\sigma))^{2}})=&\dim(V_{\chi})\Vol(X)(k^{\tau_{s}(\sigma)}_{t})(e)\\
  &+\frac{1}{2\pi}\sum_{[\gamma]\neq e} \frac{l(\gamma)\tr(\chi(\gamma))}{n_{\Gamma}(\gamma)D(\gamma)}\sum_{\sigma\in\widehat{M}}
  \overline{\tr\sigma(m_{\gamma})}\int_{\R}\Theta_{\sigma,\lambda}(k_{t}^{\tau_{s}(\sigma)})e^{-il(\gamma)\lambda}d\lambda.
\end{align*}
We continue analyzing the trace formula above in terms of characters.
We want to compute the Fourier transform $\Theta_{\sigma,\lambda}( k_{t}^{\tau_{s}(\sigma)})$ of $ k_{t}^{\tau_{s}(\sigma)}$.
%We first recall equation (4.43) from section 4.3.\\
%For $\sigma,\sigma'\in\widehat{M}$
%\begin{align}
% &\Theta_{\sigma',\lambda}(h_{t}^{\tau_{s}(\sigma)})=e^{-t\lambda^{2}} \Tr D(\sigma'), \quad \text{if}\quad \sigma'\in \{\sigma, w \sigma\}\\\notag
% &\Theta_{\sigma',\lambda}(h_{t}^{\tau_{s}(\sigma)})=0, \quad \text{if} \quad \sigma'\notin\{\sigma, w \sigma\}.
%\end{align}
%Hence, it remains to calculate the trace $\Tr D(\sigma)$.
Following \cite{MS}, we let $(\pi,\mathcal{H}_{\pi})$ be an unitary admissible representation of $G$ in a Hilbert space $\mathcal{H}_{\pi}$. We let $\mathcal{H}^{\infty}_{\pi}$ be the subspace of of smooth 
vectors of $\mathcal{H}_{\pi}$.
We set 
\begin{equation}
 \pi(K_{t}^{\tau_{s}(\sigma)}):=\int_{G}\pi(g)\otimes K_{t}^{\tau_{s}(\sigma)}(g)dg.
\end{equation}
This defines a bounded trace class operator on $\mathcal{H}_{\pi}\otimes V_{\tau_{s}(\sigma)}$. 
By \cite[p.160-161]{BM}, relative to the splitting
\begin{equation*}
 \mathcal{H}_{\pi}\otimes V_{\tau_{s}(\sigma)}=(\mathcal{H}_{\pi}\otimes V_{\tau_{s}(\sigma)})^K\oplus [(\mathcal{H}_{\pi}\otimes V_{\tau_{s}(\sigma)})^K]^{\perp},
\end{equation*}
$ \widetilde{\pi}(K_{t}^{\tau_{s}(\sigma)})$ has the form
\begin{equation}
  \widetilde{\pi}(K_{t}^{\tau_{s}(\sigma)})= \begin{pmatrix}
\pi(K_{t}^{\tau_{s}(\sigma)}) & 0 \\
 0 & 0
\end{pmatrix},
\end{equation}
with $\pi(K_{t}^{\tau_{s}(\sigma)})$ acting on $(\mathcal{H}_{\pi}\otimes V_{\tau_{s}(\sigma)})^K$.
We consider orthonormal bases $(\xi_{n}), n\in \N, (e_{j}), j=1,\cdots,k$ of the vector spaces $\mathcal{H}_{\pi}, V_{\tau_{s}}(\sigma)$, respectively, where $k:=\dim(V_{\tau_{s}(\sigma)})$.
By (5.7), we have
\begin{equation}
 \Tr(\pi(K_{t}^{\tau_{s}(\sigma)}))=\Tr(\widetilde{\pi}(K_{t}^{\tau_{s}(\sigma)})).
\end{equation}
Hence,
\begin{align}
 \Tr(\widetilde{\pi}(K_{t}^{\tau_{s}(\sigma)}))=\notag&\sum_{n}\sum_{j}\langle \widetilde{\pi}(K_{t}^{\tau_{s}(\sigma)})(\xi_n\otimes e_j),(\xi_n\otimes e_j)\rangle\\\notag
 &=\sum_{n}\sum_{j}\int_{G}\langle \pi(g)\xi_n,\xi_n\rangle\langle K_{t}^{\tau_{s}(\sigma)}(g)e_j,e_j\rangle dg\\\notag
 &=\sum_{n}\int_{G}\langle\pi(g)\xi_n,\xi_n\rangle k_{t}^{\tau_{s}(\sigma)}(g)dg\\\notag
 &=\sum_{n}\langle\pi(k_{t}^{\tau_{s}(\sigma)})\xi_n,\xi_n\rangle\\
 &=\Tr\pi(k_{t}^{\tau_{s}(\sigma)}).
 \end{align}
Let $(X_{i})_{i=1}^{d}$ be an orthonormal basis of $\mathfrak{p}$. We consider the operator acting on $(\mathcal{H}^{\infty}_{\pi}\otimes V_{\tau_{s}(\sigma)})^{K}$, defined by
\begin{equation}
 \widetilde{D}_{\tau_{s}(\sigma)}(\pi):=\sum_{i=1}^{d}X_{i}\cdot(\pi(X_{i})\otimes\Id).
\end{equation}
In \cite[p.77]{Pf}, it is proved that $ \widetilde{D}_{\tau_{s}(\sigma)}(\pi)$ maps $(\mathcal{H}^{\infty}_{\pi}\otimes V_{\tau_{s}(\sigma)})^{K}$ to $(\mathcal{H}^{\infty}_{\pi}\otimes V_{\tau_{s}(\sigma)})^{K}$.
By (5.6) we get
\begin{equation}
 \widetilde{\pi}(K_{t}^{\tau_{s}(\sigma)})=e^{-tc(\sigma)} \widetilde{D}_{\tau_{s}(\sigma)}(\pi)\circ\widetilde{\pi}(H_{t}^{\tau_{s}(\sigma)}),
\end{equation}
where 
\begin{equation*}
\widetilde{\pi}(H_{t}^{\tau_{s}(\sigma)})=\int_{G}\pi(g)\otimes H_{t}^{\tau_{s}(\sigma)}dg.
\end{equation*}
The kernel function $H_{t}^{\tau_{s}(\sigma)}$ corresponds to the integral operator
\begin{equation*}
 e^{-t(\widetilde{D}(\sigma))^{2}}f(g)=e^{-tc(\sigma)}\int_{G}H_{t}^{\tau_{s}(\sigma)}(g^{-1}g')f(g')dg',
\end{equation*}
 where $(\widetilde{D}(\sigma))^{2}$ as in (4.7). $H_{t}^{\tau_{s}(\sigma)}$ belongs to the Harish-Chandra $L^{q}$-Schwartz space
 $(\mathcal{C}^{q}(G)\otimes \End(V_{\tau_{s}(\sigma)}))^{K\times K}$.\\
As above (see equations (5.7), (5.8), (5.9)), the operator $\widetilde{\pi}(H_{t}^{\tau_{s}(\sigma)})$ relative to the splitting, 
\begin{equation*}
 \mathcal{H}_{\pi}\otimes V_{\tau_{s}(\sigma)}=(\mathcal{H}_{\pi}\otimes V_{\tau_{s}(\sigma)})^K\oplus [(\mathcal{H}_{\pi}\otimes V_{\tau_{s}(\sigma)})^K]^{\perp},
\end{equation*}
takes the form 
\begin{equation}
  \widetilde{\pi}(H_{t}^{\tau_{s}(\sigma)})= \begin{pmatrix}
\pi(H_{t}^{\tau_{s}(\sigma)}) & 0 \\
 0 & 0
\end{pmatrix},
\end{equation}
with $\pi(H_{t}^{\tau_{s}(\sigma)})$ acting on $(\mathcal{H}_{\pi}\otimes V_{\tau_{s}(\sigma)})^K$.
Then, it follows that 
\begin{equation}
 e^{t\pi(\Omega)}\Id=\pi(H_{t}^{\tau_{s}(\sigma)}),
\end{equation}
where $\Id$ denotes the identity on the space $(\mathcal{H}_{\pi}^{\infty}\otimes V_{\tau_{s}(\sigma)})^K$ (\cite[Corollary 2.2]{BM}).
We have $(\mathcal{H}_{\pi}\otimes V_{\tau_{s}(\sigma)})^K=(\mathcal{H}_{\pi}^{\infty}\otimes V_{\tau_{s}(\sigma)})^K$ and
%The operator $\widetilde{D}_{\tau_{s}(\sigma)}(\pi)$ restricted to $(\mathcal{H}_{\pi}^{\infty}\otimes V_{\tau_{s}(\sigma)})^K$
%is a trace class operator and
%we have by (5.12) and (5.14)
\begin{equation}
 \Tr(\pi(k_{t}^{\tau_{s}(\sigma)}))=e^{(\pi(\Omega)-c(\sigma))t} \Tr(\widetilde{D}_{\tau_{s}(\sigma)}(\pi)|_{(\mathcal{H}_{\pi}^{\infty}\otimes V_{\tau_{s}(\sigma)})^K}).
\end{equation}
We recall that the representation space of $\tau_{s}(\sigma)$ is given by
\begin{equation*}
 V_{\tau_{s}(\sigma)}= V_{\tau(\sigma)}\otimes S.
\end{equation*}
%of $\tau_{s}(\sigma)$ (cf. Proposition 5.1).\\
Let $\pi$ be the unitary principal series representation $\pi_{\sigma,\lambda}$ defined as in \cite[p. 7-8]{Spil}.
By \cite[Proposition 3.6]{MS}, we have for $(\sigma',V_{\sigma'})\in\widehat{M}$,
\begin{equation}
 \Tr(\widetilde{D}_{\tau_{s}(\sigma)}(\pi_{\sigma',\lambda}))=\lambda\big(\dim(V_{\sigma'}\otimes V_{\tau(\sigma)}\otimes S^{+})^{M}-\dim(V_{\sigma'}\otimes V_{\tau(\sigma)}\otimes S^{-})^{M}\big).
\end{equation}
%where $S^{+}=\Delta_{2n}^{+}$, $S^{-}=\Delta_{2n}^{-}$ denotes the representation spaces of the half spin representations $\sigma_{+},\sigma_{-}$ respectively.
Following \cite[Corollary 7.6]{Pf}, we let $\check{\sigma}'$ be the contragredient representation of $\sigma'$. 
%Then, $\check{\sigma}'=\sigma'$ if $n$ is even, and $\check{\sigma}'=w\sigma'$ if $n$ is odd (\cite[section 3.2.5]{GW}).
Since $\check{\sigma}'\cong\sigma'$, we observe by equation (4.1) in Proposition 4.1 that for $\nu_{n}>0$,
\begin{equation}
 [\dim(V_{\sigma'}\otimes V_{\tau(\sigma)}\otimes S^{+})^{M}-\dim(V_{\sigma'}\otimes V_{\tau(\sigma)}\otimes S^{-})^{M}]=[\sigma-w\sigma:\sigma'].
\end{equation}
%\begin{align}
%&,\\
%&\notag[\dim(V_{\sigma'}\otimes V_{\tau'(w\sigma)}\otimes S^{-})^{M}-\dim(V_{\sigma'}\otimes V_{\tau'(w\sigma)}\otimes S^{+})^{M}]\\\notag
%&=-[\dim(V_{\sigma'}\otimes V_{\tau'(w\sigma)}\otimes S^{+})^{M}-\dim(V_{\sigma'}\otimes V_{\tau'(w\sigma)}\otimes S^{-})^{M}]\\
%&=-[w\sigma-\sigma:\sigma']=[\sigma-w\sigma:\sigma'],
%\end{align}
%where we used the fact that the spin representation $s$ of $K$, when restricted to $M$, decomposes as $s^{+}+s^{-}$.
Since $\sigma'\in\widehat{M}$ we have that $\sigma'\in\{\sigma,w\sigma\}$, otherwise the right hand side of (5.16) vanishes.
In \cite[p. 28-29] {Spil}, the character $\Theta_{\sigma,\lambda}(q_{t}^{\sigma})$ for the case (a), is computed. 
Recall that
\begin{equation*}
 \Theta_{\sigma,\lambda}(q_{t}^{\sigma})=
 \sum_{\tau \in \widehat{K}}m_{\tau}(\sigma) \Theta_{\sigma,\lambda}(q_{t}^{\tau}), 
\end{equation*}
where
\begin{align*}
q_{t}^{\sigma}=&\sum_{\tau \in \widehat{K}}m_{\tau}(\sigma) q_{t}^{\tau},
\end{align*}
\begin{equation*}
 q_{t}^{\tau}=\tr Q_{t}^{\tau}(g),
\end{equation*} 
and $Q^{\tau}_{t}\in (\mathcal{C}^{q}(G)\otimes \End(V_{\tau}))^{K\times K}$ is the kernel associated to the operator
$e^{-t\widetilde{A}_{\tau}}$.
By \cite[equation (5.35)]{Spil}, we have
\begin{equation*}
 \Theta_{\sigma,\lambda}(q_{t}^{\sigma})=\sum_{\tau \in \widehat{K}}m_{\tau}(\sigma)
 e^{t\pi_{\sigma,\lambda}(\Omega)}[\tau\rvert_{M}:\sigma],
\end{equation*}
where 
\begin{equation*}
 \pi_{\sigma,\lambda}(\Omega)=-\lambda^{2}+c(\sigma).
\end{equation*}
This is proved in \cite[p.48]{Ar}.
We study here the case (b).
This means that by (4.5), for $\sigma'\in\widehat{M}$,
\begin{align}
& \Theta_{\sigma',\lambda}(q_{t}^{\sigma})=e^{tc(\sigma)}e^{-t\lambda^{2}},\quad\text{if}\quad \sigma'\in \{\sigma, w \sigma\},\\\notag
& \Theta_{\sigma',\lambda}(q_{t}^{\sigma})=0,\quad\text{if}\quad \sigma'\notin\{\sigma, w \sigma\}.
\end{align}
If we put together (5.14)---(5.17), we obtain
\begin{align}
 &\Theta_{\sigma',\lambda}( k_{t}^{\tau_{s}(\sigma)})=\lambda e^{-t\lambda^{2}} , \quad \text{if}\quad \sigma'=\sigma\\
 &\Theta_{\sigma',\lambda}( k_{t}^{\tau_{s}(\sigma)})=-\lambda e^{-t\lambda^{2}} , \quad \text{if}\quad \sigma'= w \sigma\\
 &\Theta_{\sigma',\lambda}( k_{t}^{\tau_{s}(\sigma)})=0, \quad \text{if} \quad \sigma'\notin\{\sigma, w \sigma\}.
\end{align}
For the identity contribution we use the fact that when $s$ is restricted to $M$ it decomposes as $s^{+}+s^{-}$.
If $ \nu_{\sigma}=(\nu_{1},\ldots,\nu_{n-1},\nu_{n})$ is the highest weight of $\sigma$,
then the highest weight of $w\sigma$ is given by $ \nu_{w\sigma}=(\nu_{1},\ldots,\nu_{n-1},-\nu_{n})$.
Specifically, for the half spin representations $s^{\pm}$ we have
\begin{equation*}
  \nu_{ws^{\pm}}=(\frac{1}{2},\ldots,\mp\frac{1}{2}).
\end{equation*}
Hence,
\begin{equation*}
 ws^{\pm}=s^{\mp}.
\end{equation*}
%Furthermore $s^{+}$ and $s^{-}$ are connected by the relation $s^{-}=ws^{+}$.
The Plancherel polynomial is an even polynomial of $\lambda$ and also $P_{s^{+}}(i\lambda)=P_{ws^{+}}(-i\lambda)=P_{s^{-}}(i\lambda)$.
Hence,
\begin{align}
 k_{t}^{\tau_{s}(\sigma)}(e)\notag&=\sum_{\sigma\in\widehat{M}}\int_{\R}\Theta_{\sigma,\lambda}( k_{t}^{\tau_{s}(\sigma)})P_{\sigma}(i\lambda)d\lambda\\
&= \int_{\R}\lambda e^{-t\lambda^{2}}P_{s^{+}}(i\lambda)d\lambda
+\int_{\R}-\lambda e^{-t\lambda^{2}}P_{s^{-}}(i\lambda)d\lambda=0.
\end{align}
For the hyperbolic contribution we use (5.18)---(5.20). 
\begin{equation*}
 \Tr(D^{\sharp}_{\chi}(\sigma)e^{-t(D^{\sharp}_{\chi}(\sigma))^{2}})=
 \frac{1}{2\pi}\sum_{[\gamma]\neq e} \frac{l(\gamma)\tr(\chi(\gamma)\otimes(\sigma(m_{\gamma})-w\sigma(m_{\gamma})))}{D(\gamma) n_{\Gamma}(\gamma)}
 \int_{\R}\lambda e^{-t\lambda^{2}}e^{-il(\gamma)\lambda}d\lambda.
\end{equation*}
Equivalently,
\begin{equation*}
 \Tr(D^{\sharp}_{\chi}(\sigma)e^{-t(D^{\sharp}_{\chi}(\sigma))^{2}})=
  \sum_{[\gamma]\neq e}\frac{-2\pi i}{(4\pi t)^{3/2}}\frac{l^{2}(\gamma)\tr(\chi(\gamma)\otimes(\sigma(m_{\gamma})-w\sigma(m_{\gamma})))}
  {n_{\Gamma}(\gamma)D(\gamma)}e^{-l^{2}(\gamma)/4t}.
\end{equation*}
All in all, we have proved the following theorems.
\begin{thm}
 For every $\sigma \in \widehat{M}$ we have for case (b)
 \begin{equation}\label{f:trace formula 3}
 \Tr(D^{\sharp}_{\chi}(\sigma)e^{-t(D^{\sharp}_{\chi}(\sigma))^{2}})=
  \sum_{[\gamma]\neq e}\frac{-2\pi i}{(4\pi t)^{3/2}}\frac{l^{2}(\gamma)\tr(\chi(\gamma)\otimes(\sigma(m_{\gamma})-w\sigma(m_{\gamma})))}
  {n_{\Gamma}(\gamma)D(\gamma)}e^{-l^{2}(\gamma)/4t}.
 \end{equation}
\end{thm}

\begin{thm}
 For every $\sigma \in \widehat{M}$ we have for case (b)
   \begin{align}\label{f:trace formula 2}
  \Tr(e^{-tA_{\chi}^{\sharp}(\sigma)})=\notag&2\dim(V_{\chi})\Vol(X)\int_{\R}e^{-t\lambda^{2}}P_{\sigma}(i\lambda)d\lambda\\
  &+\sum_{[\gamma]\neq e}\frac{l(\gamma)}{n_{\Gamma}(\gamma)}L_{sym}(\gamma;\sigma+w\sigma)
  \frac{e^{-l(\gamma)^{{2}}/4t}}{(4\pi t)^{1/2}},
 \end{align} 
where 
\begin{equation*}
 L_{sym}(\gamma;\sigma)= \frac{\tr(\sigma(m_{\gamma})\otimes\chi(\gamma))e^{-|\rho|l(\gamma)}}{\det(\Id-\Ad(m_{\gamma}a_{\gamma})_{\overline{\mathfrak{n}})}}.
\end{equation*}
\end{thm}

\begin{center}
\section{\textnormal{Meromorphic continuation of the super zeta function}}
\end{center}

Let $N\in\N$. Let $s_{i},i=1,\ldots,N$ be complex numbers such that $s_{i}\in\C-\spec(-{D^{\sharp}_{\chi}(\sigma)}^{2})$.
We consider the resolvent operator
\begin{equation*}
 R(s_{i}^{2})=({D^{\sharp}_{\chi}(\sigma)}^{2}+s_{i}^{2})^{-1}.
\end{equation*}
We want to obtain the trace class property of the operators
\begin{align*}
&\prod_{i=1}^{N}R(s_{i}^{2})\\
D^{\sharp}_{\chi}(\sigma)&\prod_{i=1}^{N}R(s_{i}^{2}).
\end{align*}
%We recall the zeta function $\zeta_{\theta}(D^{\sharp}_{\chi}(\sigma)):=\Tr(D^{\sharp}_{\chi}(\sigma)_{\theta}^{-s})$, where $s\in\C$, and $\theta\in[0,2\pi)$ is the Agmon angle for $D^{\sharp}_{\chi}(\sigma)$, which is defined 
%through the complex power of $D^{\sharp}_{\chi}(\sigma)$. Then, $\zeta_{\theta}(D^{\sharp}_{\chi}(\sigma))$ converges for $\text{\Re}(s)>d$ (Appendix A, Definition A.6).\\
In order to obtain this property, we take sufficient large $N\in\N$,
such that 
\begin{itemize}
\item for $N>\frac{d}{2}$,
\begin{equation}
 \Tr(\prod_{i=1}^{N}R(s_{i}^{2}))<\infty.
\end{equation}
\item  for $N>\frac{d}{2}+1$,
\begin{equation}
 \Tr(D^{\sharp}_{\chi}(\sigma)\prod_{i=1}^{N}R(s_{i}^{2}))<\infty.
\end{equation}
\end{itemize}
We denote the space of pseudodifferential operators of order $k$ by $\psi DO^{k}$.
To prove the trace class property of the operators above, we
observe at first that $\prod_{i=1}^{N}R(s_{i}^{2})\in\psi DO^{-2N}$.

Let $\Delta$ be the Bochner-Laplace operator with respect
to some metric, acting on $C^{\infty}(X, E_{\tau_{s}(\sigma)}\otimes E_{\chi})$.
Then, $\Delta$ is a second-order elliptic differential operator,
which is formally self-adjoint and non-negative, i.e.,
$\Delta\geq 0$. Then, by Weyl's law, we have that 
for $N>\frac{d}{2}$, 
\begin{equation*}
 (\Delta+\Id)^{-N}
\end{equation*}
is a trace class operator.
Moreover,
\begin{equation*}
 B:=(\Delta+\Id)^{N}\prod_{i=1}^{N}R(s_{i}^{2})
\end{equation*}
is $\psi DO$ of order zero.
Hence, it defines a bounded operator in $L^{2}(X, E_{\tau_{s}(\sigma)}\otimes E_{\chi})$.
Thus,
\begin{equation*}
 \prod_{i=1}^{N}R(s_{i}^{2})=(\Delta+\Id)^{-N}B
\end{equation*}
is a trace class operator.

We recall here the following expressions of the resolvents. Let $s_{1},\ldots,s_{N}\in\C$ such that 
Re$(s_{i}^{2})>-c$, for all $i=1,\ldots,N$, where $c$ is a real number such that $\spec\big(D_{\chi}^{\sharp}(\sigma)^{2}\big)\subset
\{z\in\C\colon \text{Re}(z)>c\}$.\\
Then,
\begin{align}
D^{\sharp}_{\chi}(\sigma){{(D^{\sharp}_{\chi}(\sigma)}^{2}+s_{i}^{2})}^{-1}&=\int_{0}^{\infty}e^{-ts_{i}^{2}}D^{\sharp}_{\chi}(\sigma)e^{-t{D^{\sharp}_{\chi}(\sigma)}^{2}}dt\\
(A_{\chi}^{\sharp}(\sigma)+s_{i}^{2})^{-1}&=\int_{0}^{\infty}e^{-ts_{i}^{2}}e^{-t{A_{\chi}^{\sharp}(\sigma)}}dt.
\end{align}
\begin{prop}
Let  $N\in\N$ with $N>d/2+1$.
Let $s_{1},\ldots,s_{N}\in\C$ with 
 $s_{i}\neq s_{j}$ for all $i\neq j$ such that 
\emph{Re}$(s_{i}^{2})>-c$, for all $i=1,\ldots,N$, where $c$ is a real number such that $\spec\big(D_{\chi}^{\sharp}(\sigma)^{2}\big)\subset
\{z\in\C\colon \emph{Re}(z)>c\}$.
Let $L^{s}(s):=\frac{d}{ds}\log(Z^{s}(s;\sigma,\chi))$ be the logarithmic derivative of the super zeta function. Then,
\begin{equation}
\label{eq super mer}
 \Tr(D^{\sharp}_{\chi}(\sigma)\prod_{i=1}^{N}({D^{\sharp}_{\chi}(\sigma)}^{2}+s_{i}^{2})^{-1})=-\frac{i}{2}\sum_{i=1}^{N}\bigg(\prod_{\substack{j=1\\ j\neq i}}^{N}\frac{1}{s_{j}^{2}-s_{i}^{2}}\bigg)L^{s}(s_{i}).
\end{equation}
\end{prop}
\begin{proof}
By \cite[Lemma 6.1]{Spil} and formula (6.3), we have
\begin{equation*}
 D^{\sharp}_{\chi}(\sigma)\prod_{i=1}^{N}({D^{\sharp}_{\chi}(\sigma)}^{2}+s_{i}^{2})^{-1}=\int_{0}^{\infty}\sum_{i=1}^{N}\bigg(\prod_{\substack{j=1\\ j\neq i}}^{N}\frac{1}{s_{j}^{2}-s_{i}^{2}}\bigg)e^{-ts_{i}^{2}}
   D^{\sharp}_{\chi}(\sigma)e^{-t{D^{\sharp}_{\chi}(\sigma)}^{2}}dt. 
\end{equation*}
The operators $ D^{\sharp}_{\chi}(\sigma)\prod_{i=1}^{N}({D^{\sharp}_{\chi}(\sigma)}^{2}+s_{i}^{2})^{-1}$, and
 $D^{\sharp}_{\chi}(\sigma)e^{-t{D^{\sharp}_{\chi}}}$ are both of trace class.\\
Then,
\begin{align*}
 D^{\sharp}_{\chi}(\sigma)\prod_{i=1}^{N}({D^{\sharp}_{\chi}(\sigma)}^{2}+s_{i}^{2})^{-1}&=\int_{\epsilon}^{\infty}\sum_{i=1}^{N}\bigg(\prod_{\substack{j=1\\ j\neq i}}^{N}\frac{1}{s_{j}^{2}-s_{i}^{2}}\bigg)e^{-ts_{i}^{2}}
   D^{\sharp}_{\chi}(\sigma)e^{-t{D^{\sharp}_{\chi}(\sigma)}^{2}}dt\\
&\xrightarrow[\epsilon \rightarrow 0]{R\rightarrow \infty} \int_{\epsilon}^{R}
\sum_{i=1}^{N}\bigg(\prod_{\substack{j=1\\ j\neq i}}^{N}\frac{1}{s_{j}^{2}-s_{i}^{2}}\bigg)e^{-ts_{i}^{2}}
  D^{\sharp}_{\chi}(\sigma)e^{-t{D^{\sharp}_{\chi}(\sigma)}^{2}}dt,\\
\end{align*}
%for some $\epsilon, R>0$,
where the limit is taken with respect to the trace norm 
$\lVert \mathcal{A}\rVert_{1}:=\Tr\lvert \mathcal{A}\rvert$, with
$\mathcal{A}= D^{\sharp}_{\chi}(\sigma)\prod_{i=1}^{N}({D^{\sharp}_{\chi}(\sigma)}^{2}+s_{i}^{2})^{-1}$, or $D^{\sharp}_{\chi}(\sigma)e^{-t{D^{\sharp}_{\chi}(\sigma)}^{2}}$.
We have
\begin{align*}
\Tr\Bigg(\int_{0}^{\infty}\sum_{i=1}^{N}\bigg(\prod_{\substack{j=1\\ j\neq i}}^{N}\frac{1}{s_{j}^{2}-s_{i}^{2}}\bigg)&e^{-ts_{i}^{2}}
   D^{\sharp}_{\chi}(\sigma)e^{-t{D^{\sharp}_{\chi}(\sigma)}^{2}}dt\Bigg)=\\
&\Tr\Bigg(\lim_{\substack{\epsilon \rightarrow 0\\ R\rightarrow \infty}}\int_{\epsilon}^{R}\sum_{i=1}^{N}\bigg(\prod_{\substack{j=1\\ j\neq i}}^{N}\frac{1}{s_{j}^{2}-s_{i}^{2}}\bigg)e^{-ts_{i}^{2}}
   D^{\sharp}_{\chi}(\sigma)e^{-t{D^{\sharp}_{\chi}(\sigma)}^{2}}dt\Bigg).
\end{align*}
But,
\begin{align*}
 \Tr\Bigg(\int_{\epsilon}^{R}\sum_{i=1}^{N}\bigg(\prod_{\substack{j=1\\ j\neq i}}^{N}\frac{1}{s_{j}^{2}-s_{i}^{2}}\bigg)&e^{-ts_{i}^{2}}
   D^{\sharp}_{\chi}(\sigma)e^{-t{D^{\sharp}_{\chi}(\sigma)}^{2}}dt\Bigg)=\\
&\int_{\epsilon}^{R}\sum_{i=1}^{N}\bigg(\prod_{\substack{j=1\\ j\neq i}}^{N}\frac{1}{s_{j}^{2}-s_{i}^{2}}\bigg)e^{-ts_{i}^{2}}
   \Tr(D^{\sharp}_{\chi}(\sigma)e^{-t{D^{\sharp}_{\chi}(\sigma)}^{2}})dt.
\end{align*}
Hence, it is sufficient to show that the limit
\begin{equation*}
 \lim_{\substack{\epsilon \rightarrow 0\\ R\rightarrow \infty}}\int_{\epsilon}^{R}\sum_{i=1}^{N}\bigg(\prod_{\substack{j=1\\ j\neq i}}^{N}\frac{1}{s_{j}^{2}-s_{i}^{2}}\bigg)e^{-ts_{i}^{2}}
 \Tr(D^{\sharp}_{\chi}(\sigma)e^{-t{D^{\sharp}_{\chi}(\sigma)}^{2}})dt
\end{equation*}
exists. We study the behavior of the integral in the equation above as $\epsilon\rightarrow 0$.\\
By \cite[Lemma 6.3]{Spil}, we have that as $t\rightarrow 0^{+}$
\begin{equation*}
 \sum_{i=1}^{N}\bigg(\prod_{\substack{j=1\\ j\neq i}}^{N}\frac{1}{s_{j}^{2}-s_{i}^{2}}\bigg) e^{-ts_{i}^{2}}=O(t^{N-1}).
\end{equation*}
Also, by \cite[Theorem 2.7, p.503-504]{GS}, there exists a short time asymptotic expansion of the kernel of the operator $D^{\sharp}_{\chi}(\sigma)e^{-t{D^{\sharp}_{\chi}(\sigma)}^{2}}$
\begin{equation*}
 \Tr(D^{\sharp}_{\chi}(\sigma)e^{-t{D^{\sharp}_{\chi}(\sigma)}^{2}})\sim_{t\rightarrow 0^{+}}t^{-d/2}.
\end{equation*}
%Hence, since the operator $D^{\sharp}_{\chi}(\sigma)e^{-t{D^{\sharp}_{\chi}(\sigma)}^{2}}$ is of trace class for 
%$N>d/2+1$, 
We have that as $t\rightarrow 0^{+}$
\begin{equation*}
\Bigg\lvert\sum_{i=1}^{N}\bigg(\prod_{\substack{j=1\\ j\neq i}}^{N}\frac{1}{s_{j}^{2}-s_{i}^{2}}\bigg)e^{-ts_{i}^{2}}
 \Tr(D^{\sharp}_{\chi}(\sigma)e^{-t{D^{\sharp}_{\chi}(\sigma)}^{2}}) \Bigg\lvert\leq Ct,
\end{equation*}
where $C$ is a positive constant.
All in all, we have proved
 \begin{equation*}
   \Tr(D^{\sharp}_{\chi}(\sigma)\prod_{i=1}^{N}({D^{\sharp}_{\chi}(\sigma)}^{2}+s_{i}^{2})^{-1})=\int_{0}^{\infty}\sum_{i=1}^{N}\bigg(\prod_{\substack{j=1\\ j\neq i}}^{N}\frac{1}{s_{j}^{2}-s_{i}^{2}}\bigg)e^{-ts_{i}^{2}}
   \Tr(D^{\sharp}_{\chi}(\sigma)e^{-t{D^{\sharp}_{\chi}(\sigma)}^{2}})dt.
 \end{equation*}
We apply now the trace formula (5.22) for the operator $D^{\sharp}_{\chi}(\sigma)e^{-t{D^{\sharp}_{\chi}(\sigma)}^{2}}$. Then, we get
\begin{align}
 \Tr(D^{\sharp}_{\chi}(\sigma)\prod_{i=1}^{N}({D^{\sharp}_{\chi}(\sigma)}^{2}+s_{i}^{2})^{-1})\notag&=\int_{0}^{\infty}\sum_{i=1}^{N}\bigg(\prod_{\substack{j=1\\ j\neq i}}^{N}\frac{1}{s_{j}^{2}-s_{i}^{2}}\bigg)e^{-ts_{i}^{2}}\\
  & \bigg\{\sum_{[\gamma]\neq e}\frac{-2\pi i}{(4\pi t)^{3/2}}\frac{l^{2}(\gamma)\tr(\chi(\gamma)\otimes(\sigma(m_{\gamma})-w\sigma(m_{\gamma})))}
  {n_{\Gamma}(\gamma)D(\gamma)}e^{-l^{2}(\gamma)/4t}\bigg\}.
\end{align}
If we use the formula (see \cite[p.146, (28)]{Er})
\begin{equation*}
 \int_{0}^{\infty}e^{-ts^{2}}\frac{1}{(4\pi t)^{3/2}}e^{-l^{2}(\gamma)/4t}dt=\frac{e^{-l(\gamma)s}}{4\pi l(\gamma)}
\end{equation*}
equation (6.6) becomes
\begin{align*}
\Tr(D^{\sharp}_{\chi}(\sigma)\prod_{i=1}^{N}({D^{\sharp}_{\chi}(\sigma)}^{2}+s_{i}^{2})^{-1})=&\frac{-i}{2}\sum_{i=1}^{N}\bigg(\prod_{\substack{j=1\\ j\neq i}}^{N}\frac{1}{s_{j}^{2}-s_{i}^{2}}\bigg)\\
&\bigg\{\sum_{[\gamma]\neq e}\frac{l(\gamma)\tr(\chi(\gamma)\otimes(\sigma(m_{\gamma})-w\sigma(m_{\gamma})))}
{n_{\Gamma}(\gamma)D(\gamma)}e^{-l(\gamma)s_{i}}\bigg\}.
\end{align*}
Hence, by equation \eqref{f:log der super} we get
\begin{equation*}
 \Tr(D^{\sharp}_{\chi}(\sigma)\prod_{i=1}^{N}({D^{\sharp}_{\chi}(\sigma)}^{2}+s_{i}^{2})^{-1})=\frac{-i}{2}\sum_{i=1}^{N}\bigg(\prod_{\substack{j=1\\ j\neq i}}^{N}\frac{1}{s_{j}^{2}-s_{i}^{2}}\bigg)L^{s}(s_{i}).
\end{equation*}
\end{proof}
The meromorphic continuation of the super zeta function follows from the Proposition (6.1) above.
\begin{thm}
The super zeta function $Z^{s}(s;\sigma,\chi)$ admits a meromorphic continuation to the whole complex plane $\C$. The singularities 
are located at $\{s_{k}^{\pm}=\pm i\lambda_{k}:\lambda_{k}\in \spec(D^{\sharp}_{\chi}(\sigma)), k\in\N\}$
of order  $\pm m_{s}(\lambda_{k})$, where $m_{s}(\lambda_{k})=m(\lambda_{k})-m(-\lambda_{k})\in\N$ and 
$m(\pm\lambda_{k})$ denotes the algebraic multiplicity of the eigenvalue $\pm\lambda_{k}$.
\end{thm}
\begin{proof}
We define the function $\Phi(s_{1},s_{2},\ldots,s_{N})$ of the complex variables $s_{1},s_{2},\ldots,s_{N}$ by
\begin{equation}
\Phi(s_{1},s_{2},\ldots,s_{N})=-\frac{i}{2}\sum_{i=1}^{N}\bigg(\prod_{\substack{j=1\\ j\neq i}}^{N}\frac{1}{s_{j}^{2}-s_{i}^{2}}\bigg)L^{s}(s_{i}).
\end{equation}
By Lidskii's theorem and \cite[Lemma 6.1]{Spil}, (6.5) becomes
 \begin{equation}
\sum_{\lambda_{k}}\sum_{i=1}^{N}\bigg(\prod_{\substack{j=1\\ j\neq i}}^{N}\frac{1}{s_{j}^{2}-s_{i}^{2}}\bigg) m_{s}(\lambda_{k})\lambda_{k}\frac{1}{(\lambda_{k})^{2}+s_{i}^{2}}=\Phi(s_{1},s_{2},\ldots,s_{N}).
 \end{equation}
%where $\lambda_{k}\in\spec(D^{\sharp}_{\chi}(\sigma))$, and $m_{s}(\lambda_{k})$ are integer numbers, which are 
%equal to the algebraic multiplicities of the corresponding eigenvalues
%$\lambda_{k}$.\\
We fix the complex numbers $s_{i}, i=2,\ldots,N$ with $s_{i}\neq s_{j}$ for $i,j=2,\ldots, N$
and let the complex number $s=s_{1}$ vary.
Hence,
\begin{equation*}
\Phi(s,s_{2},\ldots,s_{N})=\Phi(s)
\end{equation*}
The term that contains the logarithmic derivative $L^{s}(s)$ in $\Phi(s)$ is of the form
\begin{equation}
 \frac{-i}{2}\bigg(\prod_{j=2}^{N}\frac{1}{s_{j}^{2}-s^{2}}\bigg)L^{s}(s).
\end{equation}
The term of 
\begin{equation*}
\sum_{\lambda_{k}}\sum_{i=1}^{N}\bigg(\prod_{\substack{j=1\\ j\neq i}}^{N}\frac{1}{s_{j}^{2}-s_{i}^{2}}\bigg) m_{s}(\lambda_{k})\lambda_{k}\frac{1}{(\lambda_{k})^{2}+s_{i}^{2}}
\end{equation*}
which is singular at $s=\pm i\lambda_{k}$, $k\in \N$ is 
\begin{equation*}
\bigg(\prod_{j=2}^{N}\frac{1}{s_{j}^{2}-s^{2}}\bigg) m_{s}(\lambda_{k})\lambda_{k}\frac{1}{(\lambda_{k})^{2}+s^{2}}.
\end{equation*}
If we multiply both sides of  (6.8) by
\begin{equation*}
 2i\prod_{j=2}^{N}(s_{j}^{2}-s^{2}),
\end{equation*}
we see that the residue of the logarithmic derivative $L^{s}(s)$ at $\pm i\lambda_{k}$ is $\pm m_{s}(\lambda_{k})$.

By (3.8), $L^{s}(s)$ decreases exponentially as Re$(s)\rightarrow \infty$. Hence, the integral
\begin{equation*}
 \int_{s}^{\infty}L^{s}(w)dw
\end{equation*}
over a path connecting $s$ and infinity
is well defined and 
\begin{equation}
 \log Z^{s}(s;\sigma,\chi)=-\int_{s}^{\infty}L^{s}(w)dw.
\end{equation}
The integral above depends on the choice of the path, because $L^{s}(s)$
has singularities at $s_{k}^{\pm}$.
%, namely poles of integer residues.
Nevertheless, since all the
residues of the singularities are integers, it follows that the exponential of the integral in 
the right hand side of (6.10)
is independent of the choice of the path.
%To treat this dependence, we exponentiate 
%the right-hand side of (6.17) and 
The meromorphic continuation of the super zeta function
$Z^{s}(s;\sigma,\chi)$ to the whole complex plane follows.
 
%Then, the logarithmic derivative 
% by basic theorem of complex analysis, we can imply that $Z^{s}(s;\chi,\sigma)$ has singularities 
%located exactly at   $s_{k}^{+}$ of order $m_{s}(\lambda_{k})$.

%By equation (8.1.7) we observe also that the logarithmic derivative $L^{s}(s)$ might have poles of first
%order at $t^{\pm}=\pm s_{j}$, $j=2,\ldots N$, which are arising from the
%terms that include the finite products.
%Although, if we calculate the residues at these points, we see that they vanish.
%\begin{equation}
% Res_{s_{1}\rightarrow t^{\pm}}\sum_{i=2}^{N}\bigg(\prod_{\substack{j=1\\ j\neq i}}^{N}\frac{1}{s_{j}^{2}-s_{i}^{2}}\bigg)=\lim_{s_{1}\rightarrow t^{\pm}}(s_{1}-t^{\pm})\sum_{i=2}^{N}\bigg(\prod_{\substack{j=1\\ j\neq i}}^{N}\frac{1}{s_{j}^{2}-s_{i}^{2}}\bigg)=0.
%\end{equation}
%The assertion follows.
\end{proof}

\begin{center}
{\section{\textnormal{Meromorphic continuation of the symmetrized zeta function}}}
\end{center}

Let  $N\in\N$ with $N>d/2$.
We choose  $s_{1},\ldots,s_{N}\in\C$ with 
 $s_{i}\neq s_{j}$ for all $i\neq j$ such that 
 Re$(s_{i}^{2})>-r$, for all $i=1,\ldots,N$, where $r$ is a real number such that $\spec\big(A_{\chi}^{\sharp}(\sigma)\big)\subset
\{z\in\C\colon \text{Re}(z)>r\}$.\\
Then, by \cite[Lemma 6.1]{Spil} and equation (6.4), we have
\begin{equation*}
  \prod_{i=1}^{N}(A_{\chi}^{\sharp}(\sigma)+s_{i}^{2})^{-1}=\int_{0}^{\infty}\sum_{i=1}^{N}\bigg(\prod_{\substack{j=1\\ j\neq i}}^{N}\frac{1}{s_{j}^{2}-s_{i}^{2}}\bigg)e^{-ts_{i}^{2}}
e^{-t{A_{\chi}^{\sharp}(\sigma)}}dt.
\end{equation*}
As in the proof of Proposition 6.1, we can consider the trace of the operators in the formula above and get
\begin{equation*}
  \Tr \prod_{i=1}^{N}(A_{\chi}^{\sharp}(\sigma)+s_{i}^{2})^{-1}=\int_{0}^{\infty}\sum_{i=1}^{N}\bigg(\prod_{\substack{j=1\\ j\neq i}}^{N}\frac{1}{s_{j}^{2}-s_{i}^{2}}\bigg)e^{-ts_{i}^{2}}\Tr 
e^{-t{A_{\chi}^{\sharp}(\sigma)}}dt.
\end{equation*}
We insert the trace formula (5.23) for the operator $e^{-tA_{\chi}^{\sharp}(\sigma)}$ and get
\begin{align}
\notag\Tr \prod_{i=1}^{N}(A_{\tau,\chi}^{\sharp}&(\sigma)+s_{i}^{2})^{-1}=\\\notag&\int_{0}^{\infty}\sum_{i=1}^{N}\bigg(\prod_{\substack{j=1\\ j\neq i}}^{N}\frac{1}{s_{j}^{2}-s_{i}^{2}}\bigg)e^{-ts_{i}^{2}}
\bigg\{2\dim(V_{\chi})\Vol(X)\int_{\R}e^{-t\lambda^{2}}P_{\sigma}(i\lambda)d\lambda\\
&+\sum_{[\gamma]\neq e}\frac{l(\gamma)}{n_{\Gamma}(\gamma)}L_{sym}(\gamma;\sigma+w\sigma)
\frac{e^{-l(\gamma)^{{2}}/4t}}{(4\pi t)^{1/2}}\bigg\}dt.
\end{align}
The first sum in the right hand side of (7.1) includes the double integral
\begin{equation*}
 I=\int_{0}^{\infty}\int_{\R}\sum_{i=1}^{N}\bigg(\prod_{\substack{j=1\\ j\neq i}}^{N}\frac{1}{s_{j}^{2}-s_{i}^{2}}\bigg)
e^{-ts_{i}^{2}}e^{-t\lambda^{2}}P_{\sigma}(i\lambda)d\lambda dt,
\end{equation*}
which has been computed in \cite[p. 33-34]{Spil}
\begin{align*}
I=\sum_{i=1}^{N}\bigg(\prod_{\substack{j=1\\ j\neq i}}^{N}\frac{1}{s_{j}^{2}-s_{i}^{2}}\bigg)
\frac{\pi}{s_{i}}P_{\sigma}(s_{i}).
\end{align*}
Hence, equation (7.1) reads
\begin{align*}
 \Tr \prod_{i=1}^{N}(A_{\chi}^{\sharp}(\sigma)+s_{i}^{2})^{-1}&=\sum_{i=1}^{N}\bigg(\prod_{\substack{j=1\\ j\neq i}}^{N}\frac{1}{s_{j}^{2}-s_{i}^{2}}\bigg)\frac{\pi}{s_{i}} 2\dim(V_{\chi})\Vol(X)P_{\sigma}(s_{i})\\\notag
&+\sum_{i=1}^{N}\frac{1}{2s_{i}}\bigg(\prod_{\substack{j=1\\ j\neq i}}^{N}\frac{1}{s_{j}^{2}-s_{i}^{2}}\bigg)\sum_{[\gamma]\neq[e]} \frac{l(\gamma)}{n_{\Gamma}(\gamma)}L_{sym}(\gamma;\sigma+w\sigma)
e^{-s_{i}l(\gamma)}.
\end{align*}
By \eqref{f:log der symmetrized}, we can insert the logarithmic derivative $L_{S}(s)$ of the symmetrized zeta function.
Then, we get
\begin{align}
\label{eq mer sym}
 \Tr \prod_{i=1}^{N}(A_{\chi}^{\sharp}(\sigma)+s_{i}^{2})^{-1}\notag&=\sum_{i=1}^{N}\bigg(\prod_{\substack{j=1\\ j\neq i}}^{N}\frac{1}{s_{j}^{2}-s_{i}^{2}}\bigg)\frac{\pi}{s_{i}} 2\dim(V_{\chi})\Vol(X)P_{\sigma}(s_{i})\\
&+\sum_{i=1}^{N}\bigg(\prod_{\substack{j=1\\ j\neq i}}^{N}\frac{1}{s_{j}^{2}-s_{i}^{2}}\bigg)\frac{1}{2s_{i}}L_{S}(s_{i}).
\end{align}
Equation (7.2) will give the meromorphic continuation of the symmetrized zeta function.
\begin{thm}
The symmetrized zeta function $S(s;\sigma,\chi)$ admits a meromorphic continuation to the whole complex plane $\C$. The set of the singularities
equals $\{s_{k}^{\pm}=\pm i{\sqrt{\mu_{k}}}:\mu_{k}\in \spec(A^{\sharp}_{\chi}(\sigma)), k\in\N\}$.
The orders of the singularities are equal to $m(\mu_{k})$, where $m(\mu_{k})\in\N$ denotes the algebraic multiplicity of the eigenvalue $\mu_{k}$.
For $\mu_{0}=0$, the order of the singularity $s_{0}$ is equal to $2m(0)$.
\end{thm}
\begin{proof}
By \cite[Lemma 6.1]{Spil} and equation (7.2), we get
\begin{align}
\sum_{\mu_{k}}\sum_{i=1}^{N}\bigg(\prod_{\substack{j=1\\ j\neq i}}^{N}\frac{1}{s_{j}^{2}-s_{i}^{2}}\bigg)\frac{m(\mu_{k})}{\mu_{k}+s_{i}^{2}}=\notag&\sum_{i=1}^{N}\bigg(\prod_{\substack{j=1\\ j\neq i}}^{N}\frac{1}{s_{j}^{2}-s_{i}^{2}}\bigg)\frac{\pi}{s_{i}}2\dim(V_{\chi})\Vol(X)P_{\sigma}(s_{i})\\\notag
&+\sum_{i=1}^{N}\bigg(\prod_{\substack{j=1\\ j\neq i}}^{N}\frac{1}{s_{j}^{2}-s_{i}^{2}}\bigg)\frac{1}{2s_{i}}L_{S}(s_{i}).
\end{align}
We multiply the last equation by $2s_{1}$.
\begin{align}
\sum_{\mu_{k}}\sum_{i=1}^{N}\bigg(\prod_{\substack{j=1\\ j\neq i}}^{N}\frac{1}{s_{j}^{2}-s_{i}^{2}}\bigg)2s_{1}\frac{m(\mu_{k})}{\mu_{k}+s_{i}^{2}}=\notag&\sum_{i=1}^{N}\bigg(\prod_{\substack{j=1\\ j\neq i}}^{N}\frac{1}{s_{j}^{2}-s_{i}^{2}}\bigg)\frac{4\pi s_{1}}{s_{i}}\dim(V_{\chi})\Vol(X)P_{\sigma}(s_{i})\\
&+\sum_{i=1}^{N}\bigg(\prod_{\substack{j=1\\ j\neq i}}^{N}\frac{1}{s_{j}^{2}-s_{i}^{2}}\bigg)\frac{s_{1}}{s_{i}}L_{S}(s_{i}).
\end{align}
We define the function $\varXi(s_{1},\ldots,s_{N})$ of the complex variables $s_{1},\ldots,s_{N}$ by
\begin{equation*}
\varXi(s_{1},\ldots,s_{N}):=\sum_{i=1}^{N}\bigg(\prod_{\substack{j=1\\ j\neq i}}^{N}\frac{1}{s_{j}^{2}-s_{i}^{2}}\bigg)\frac{s_{1}}{s_{i}}L_{S}(s_{i}).
\end{equation*}
We fix the complex numbers $s_{i}, i=2,\ldots,N$ with $s_{i}\neq s_{j}$ for $i,j=2,\ldots, N$
and let the complex number $s=s_{1}$ vary.
Then, 
\begin{equation*}
 \varXi(s,\ldots,s_{N})=\varXi(s),
\end{equation*}
and equation (7.3) becomes
\begin{align}
\notag\sum_{\mu_{k}}\sum_{i=1}^{N}\bigg(\prod_{\substack{j=1\\ j\neq i}}^{N}\frac{1}{s_{j}^{2}-s_{i}^{2}}\bigg)2s\frac{m(\mu_{k})}{\mu_{k}+s_{i}^{2}}
=\sum_{i=1}^{N}&\bigg(\prod_{\substack{j=1\\ j\neq i}}^{N}\frac{1}{s_{j}^{2}-s_{i}^{2}}\bigg)\frac{4\pi s}{s_{i}}\dim(V_{\chi})\Vol(X)P_{\sigma}(s_{i})\\
+&\varXi(s).
\end{align}
The term that contains the logarithmic derivative $L_{S}(s)$ in $\varXi(s)$ is of the form
\begin{equation}
 \bigg(\prod_{j=2}^{N}\frac{1}{s_{j}^{2}-s^{2}}\bigg)L_{S}(s).
\end{equation}
The term of 
\begin{equation*}
 \sum_{\mu_{k}}\sum_{i=1}^{N}\bigg(\prod_{\substack{j=1\\ j\neq i}}^{N}\frac{1}{s_{j}^{2}-s_{i}^{2}}\bigg)2s\frac{m(\mu_{k})}{\mu_{k}+s_{i}^{2}},
\end{equation*}
which is singular at  $s=\pm i\sqrt{\mu_{k}}$, $k\in \N$ is 
\begin{equation*}
\bigg(\prod_{j=2}^{N}\frac{1}{s_{j}^{2}-s^{2}}\bigg)2s\frac{m(\mu_{k})}{\mu_{k}+s^{2}}.
\end{equation*}
%Then the singularities of the right hand 
%side are first order
%poles at $s=\pm i\sqrt{\mu_{k}}$, $k\in \N$. The term on the right-hand side 
%that is singular at
%$s=\pm i\sqrt{\mu_{k}}$ is
%\[
%\frac{2m(\mu_k)}{\mu_k+s^2}.
%\]
%The term of the sum for $i=1$ in the left-hand side of (6.31) gives simple poles at
% $s_{k}^{\pm}=\pm i\sqrt{\mu_{k}}$.
%We calculate the residues.
%\begin{align*}
 %Res_{s=s_{k}}\varXi(s)=
%&\lim_{s\rightarrow s_{k}^{\pm}}(s-s_{k}^{\pm})\sum_{\mu_{k}}\sum_{i=1}^{N}\bigg(\prod_{\substack{j=1\\ j\neq i}}^{N}\frac{1}{s_{j}^{2}-s^{2}}\bigg)2s\frac{m(\mu_{k})}{\mu_{k}+s_{i}^{2}}\\
%&=\bigg(\prod_{j=2}^{N}\frac{1}{s_{j}^{2}\pm\mu_{k}}\bigg)m(\mu_{k}).
%\end{align*}
%If now $\mu_{0}=0$, then we calculate the residue at $s_{0}=0$.
%\begin{align*}
 %Res_{s=0}\varXi(s)=&\lim_{s\rightarrow \mu_{0}}
%s\sum_{\mu_{k}}\sum_{i=1}^{N}\bigg(\prod_{\substack{j=1\\ j\neq i}}^{N}\frac{1}{s_{j}^{2}-s_{i}^{2}}\bigg)2s\frac{m(\mu_{k})}{\mu_{k}+s_{i}^{2}}\\
%&=\bigg(\prod_{j=2}^{N}\frac{1}{s_{j}^{2}}\bigg)2m(0).
%\end{align*}
%Hence, if we multiply (6.31) by the finite product
%\begin{equation*}
 %\prod_{j=2}^{N}(s_{j}^{2}-s^{2}),
%\end{equation*}
We multiply both sides of the 
equality (7.4) by
\begin{equation*}
 \prod_{j=2}^N(s_j^2-s^2).
\end{equation*}
Then, the residues of $L_{S}(s)$ at the points $\pm i\sqrt{\mu_{k}}$ 
are $m(\mu_{k})$, for $k\neq0$ and $2m(0)$, for $k=0$.
%\begin{equation*}
% \bigg(\prod_{j=2}^{N}\frac{1}{s_{j}^{2}\pm\mu_{k}}\bigg)m(\mu_{k}),
%\end{equation*}
%and for $k=0$, they are equal to
%\begin{equation*}
% \bigg(\prod_{j=2}^{N}\frac{1}{s_{j}^{2}}\bigg)2m(0).
%\end{equation*}
%Hence, for $k\neq 0$, the logarithmic derivative $L_{S}(s)$
%has simple poles at $s_{k}^{\pm}=\pm i\sqrt{\mu_{k}}$ of residues $m(\mu_{k})$, 
%which are positive integers.
%For $k=0$, it has a simple pole at $s_{0}=\mu_{0}=0$ of residue $2m(0)\in\N$. 

By (3.7), $L_{S}(s)$ decreases exponentially as Re$(s)\rightarrow \infty$. Therefore, the integral
\begin{equation*}
 \int_{s}^{\infty}L_{S}(w)dw
\end{equation*}
over a path connecting $s$ and infinity is well defined and 
\begin{equation}
 \log S(s;\sigma,\chi)=-\int_{s}^{\infty}L_{S}(w)dw.
\end{equation}
The integral above depends on the choice of the path, because $L_{S}(s)$
has singularities.
Since the residues of the singularities 
are integers, we can use the same argument as in the proof of Theorem 6.2.
If we exponentiate the right hand side of (7.6), then this exponential 
is independent of the choice of the path. The meromorphic continuation of the 
symmetrized zeta function $S(s;\sigma,\chi)$ to the whole complex plane follows.
\end{proof}

\begin{center}
\section{\textnormal{Meromorphic continuation of the Selberg and Ruelle zeta function}}
\end{center}

\begin{thm}
The Selberg zeta function $Z(s;\sigma,\chi)$ admits a meromorphic continuation to the whole complex plane $\C$. The set of the singularities
equals to $\{s_{k}^{\pm}=\pm i\lambda_{k}:\lambda_{k}\in \spec(D^{\sharp}_{\chi}(\sigma)),k\in\N\}$.
The orders of the singularities are equal to $\frac{1}{2}(\pm m_{s}(\lambda_{k})+m(\lambda_{k}^{2})$.
For $\lambda_{0}=0$, the order of the singularity is equal to $m(0)$.
\end{thm}
\begin{proof}
 We observe at first that
\begin{equation*}
  Z(s;\sigma,\chi)=\sqrt{S(s;\sigma,\chi)Z^{s}(s;\sigma,\chi)}.
 \end{equation*}
Recall that by equation (4.8) we have $A^{\sharp}_{\chi}(\sigma)=(D^{\sharp}_{\chi}(\sigma))^{2}$. Hence, we can identify the eigenvalues $\mu_{k}$
of $A^{\sharp}_{\chi}(\sigma)$ with $\lambda_{k}^{2}$, where $\lambda_{k}\in \spec(D^{\sharp}_{\chi}(\sigma))$.
By Theorem 6.2 and Theorem 7.1, the product $S(s;\sigma,\chi)Z^{s}(s,\sigma,\chi)$ has its singularities at $s_{k}^{\pm}=\pm i\lambda_{k}$,
of order $\pm m_{s}(\lambda_{k})+m(\lambda_{k}^{2})$.
We need to prove that the order of the singularities of $Z(s;\sigma,\chi)$ is an even integer.
This follows from the definition of the algebraic multiplicities $m_{s}(\lambda_{k}),m(\lambda_{k}^{2})$
and the constuction of the locally homogenous vector bundles $E(\sigma),E_{\tau_{s}(\sigma)}$ associated to the representations $\tau(\equiv\tau_{\sigma})$ and
$\tau_{s}(\sigma)$ of $K$.
By \cite[ Proposition 5.4]{Spil} together with equation (4.5), and Proposition 4.1 together with equation (4.2),
we can choose representations $\tau$ of $K$, such that 
%$i^{*}(\tau)=\sigma+w\sigma$ and 
$E(\sigma)=E_{\tau_{s}(\sigma)}$ up to a $\Z_{2}$-grading.
%Then, the twisted Bochner-Laplace operators $A_{\chi}^{\sharp}(\sigma)$ associated to the representations $\tau,\tau_{s}(\sigma)$, respectively, coincide.
Hence, $
 m_{s}(\lambda_{k})\equiv m(\lambda_{k}^{2})\mod 2.
$
The assertion follows.
\end{proof}
\begin{thm}
For every $\sigma\in\widehat{M}$, the Ruelle zeta function $R(s;\sigma,\chi)$ admits a meromorphic continuation to the whole complex plane $\C$.
\end{thm}
\begin{proof}
The assertion follows from Theorem 8.1 together with \cite[Theorem 6.6]{Spil}.
\end{proof}

\bibliographystyle{amsalpha}
\bibliography{references3}
UNIVERSIT\"{A}T BONN, 		%	\\
MATHEMATISCHES INSTITUT, 	%	\\
ENDENICHER ALLE 60, 		%	\\
D-53115,			%	\\
GERMANY					\\
\textit{E-mail address}: xspilioti@math.uni-bonn.de

\end{document}